\documentclass[12pt]{article}        
        
\usepackage{latexsym,amsfonts,amsmath,epsfig,tabularx,amsthm,dsfont,enumerate} 
        
\theoremstyle{definition}        
\newtheorem{thm}{Theorem}        
\newtheorem{rem}[thm]{Remark}        
        
\newtheorem{prop}[thm]{Proposition}          
\newtheorem{cor}[thm]{Corollary}

\newcommand\reals{\mathbb{R}}        
\newcommand\R{{\mathbb{R}}}        
\newcommand\C{{\mathbb{C}}}        
\renewcommand\natural{\mathbb{N}}        
\newcommand\nat{\mathbb{N}}        
\newcommand\N{\mathbb{N}}        
\newcommand\Z{\mathbb{Z}}

\newcommand\eps{\varepsilon}        
\renewcommand\a{\alpha}

\newcommand\eu{{\rm e}}

\newcommand{\e}{\varepsilon}        
\newcommand{\wt}{\widetilde}        
\newcommand{\widebar}[1]{\mbox{\kern1pt\hbox{\vbox{\hrule height 0.5pt \kern0.25ex        
        \hbox{\kern-0.05em \ensuremath{#1 }\kern-0.05em}}}}\kern-0.1pt}        
\newcommand\il{\left<}        
\newcommand\ir{\right>}        
        
\newcommand{\dx}{\, {\rm d} x }        
\newcommand{\sgn}{\operatorname{sgn}}         
\newcommand{\abs}[1]{\left\vert #1 \right\vert}        
\newcommand{\norm}[1]{\left\Vert #1 \right\Vert}

\newlength{\fixboxwidth}        
\title{Complexity of Oscillatory Integration        
\\ for Univariate Sobolev Spaces}        
\author{Erich Novak\footnote{This        
author was partially supported by the DFG-Priority Program 1324.},         
Mario Ullrich\footnote{This        
author was partially supported by the DFG GRK 1523.}\\        
Mathematisches Institut, Universit\"at Jena\\        
Ernst-Abbe-Platz 2, 07743 Jena, Germany\\        
email: erich.novak@uni-jena.de, ullrich.mario@gmail.com\\        
\qquad        
Henryk Wo\'zniakowski\footnote{This author was partially        
supported by the National Science        
Foundation and by the National Science Centre, Poland, 
based on the decision DEC-2013/09/B/ST1/04275.  
}\\     
Department of Computer Science, Columbia University,\\        
New York, NY 10027, USA, and\\        
Institute of Applied Mathematics, University of Warsaw\\        
ul. Banacha 2, 02-097 Warszawa, Poland\\        
email:\ henryk@cs.columbia.edu}        
        
\begin{document}        
\maketitle        
\begin{abstract}      
We analyze univariate oscillatory integrals for       
the standard Sobolev spaces      
$H^s$ of periodic and non-periodic      
functions with an arbitrary integer $s\ge1$.      
We find matching lower and upper bounds on the minimal worst case      
error of algorithms that use~$n$ function or derivative values.      
We also find sharp bounds on the information complexity       
which is the minimal $n$ for which the       
absolute or normalized error is at most~$\e$.      
We show surprising relations between the       
information complexity and the oscillatory weight.      
We also briefly consider the case of $s=\infty$.     
\end{abstract}      
      
\section{Introduction}        
 
We study the approximate computation of univariate oscillatory integrals      
(Fourier coefficients)  
\begin{equation}     \label{eq:problem}  
I_k (f) = \int_0^1 f(x) \,\eu^{-2\pi\,i\,kx} \, {\rm d} x,        
\qquad i=\sqrt{-1},        
\end{equation}         
where $k \in \Z$ and $f \in H^s$.  
We improve the known upper bounds and also prove matching lower bounds,  
i.e., we study the complexity of this computational problem.  
By $H^s$ we mean the standard Sobolev (Hilbert) space;  
we study spaces of periodic and non-periodic     
functions defined on $[0,1]$       
with an arbitrary integer $s\ge1$.  
We usually consider a finite $s$ but we also briefly consider     
the case of $s=\infty$. Although we consider arbitrary      
integers~$k$, our emphasis is for large $|k|$ and  
we explain our results here       
only for such $k$.      
      
We compute  
the initial error (the norm of $I_k$)  
as well as the worst case error of our algorithms  
exactly. This is possible since  
we assume that $k$ is an integer.  
For the periodic case the initial error is of order $|k|^{-s}$,        
whereas for the non-periodic case it is independent of $s$ and is      
roughly $|k|^{-1}$. This means that the initial error for the       
periodic case is much smaller for large $s$.       
For $s=\infty$, the periodic case  
leads to the space of only constant functions     
and the problem becomes trivial since  
the initial error is zero for all $k\not=0$.     
The non-periodic case is still reasonable with the  
initial error roughly $|k|^{-1}$.       
      
For a finite $s$ and the periodic case, we prove that       
an  algorithm that uses $n$ function values at equally      
spaced points is nearly optimal,  
and its worst case error is bounded by $C_s(n+|k|)^{-s}$       
with an exponentially small $C_s$ in $s$.       
For the non-periodic case, we first compute      
successive derivatives up to order $s-1$ at the end-points      
$x=0$ and $x=1$. These derivatives values are used to periodize the      
function and this allows us to obtain similar error bounds like for the      
periodic case.       
Asymptotically in $n$, the worst case error      
of the algorithm is of order $n^{-s}$ independent of $k$       
for both periodic and non-periodic cases.       
      
Near optimality of this algorithm is shown by proving a lower bound       
of order $(n+|k|)^{-s}$ which holds for all algorithms that use       
the values of function and derivatives up to order $s-1$ at $n$      
arbitrarily chosen points from $[0,1]$. We     
establish the lower bound by constructing a periodic function that     
vanishes with all its derivatives up to order $s-1$ at the points       
sampled by a given algorithm,     
belongs to the unit ball       
of the space $H^s$, and its oscillatory integral is of order      
$(n+|k|)^{-s}$.       
      
For $s=\infty$, we provide two  
algorithms which compute successive derivatives     
and/or function values at equally spaced points.     
The worst case error of one of these algorithms is      
super exponentially small in $n$.      
For $s=\infty$,    
we do not have a matching lower bound.     
       
We consider the absolute and normalized error criteria.      
For the absolute error criterion, we want to find the information      
complexity which is defined as the smallest $n$ for      
which the $n$th minimal error is at most $\e\in(0,1)$, whereas for the      
normalized error criterion, the information complexity is       
the smallest $n$ for which      
the $n$th minimal error reduces the initial error by a factor $\e$.       
For a finite $s$ we obtain the following results.      
\begin{itemize}      
\item       
For the absolute error criterion and the periodic case, the      
information complexity is      
zero if $\e>1/(2\pi|k|)^s$ and otherwise is roughly      
$\e^{-1/s}-|k|$. This means that in this case       
the problem becomes easier for large $|k|$.       
\item For the normalized error criterion and for the periodic case,       
the information complexity  is       
of order $|k|\,\e^{-1/s}$. Hence, in this case the problem becomes      
harder for larger $|k|$.       
\item       
For the absolute error criterion and the non-periodic case,      
the information complexity is  
zero if $\e\ge1.026/(2\pi|k|)$ and otherwise is      
roughly lower bounded by $\e^{-1/s}-|k|$ and upper bounded by      
$\e^{-1/s} +2s-1-|k|$. As for the periodic case, the problem becomes      
easier for large $k$.        
\item      
For the normalized error criterion and the non-periodic case,      
the information complexity is of order $|k|^{1/s}\,\e^{-1/s}$ for      
very small $\e$.       
In this case, the dependence on $|k|$ is more lenient than for      
the periodic case especially if $s$ is large.      

The dependence on $|k|$ is quite intriguing if      
$|k|$ goes to infinity.       
For $s=1$ and fixed~$\e$, the information complexity goes to      
infinity linearly with $|k|$. However, the situation is quite      
different for $s\ge2$. Then       
for large $|k|$ the information complexity is bounded by $2s$     
if $\e$ is fixed or if $\e$ tends to zero like $|k|^{-\eta}$      
with $\eta\in(0,s-1)$.      
\end{itemize}      
     
For $s=\infty$, we obtain only upper bounds on the information complexity.     
For $\e$ tending to zero they are roughly $\ln(\e^{-1})/\ln(\ln(\e^{-1}))$      
independent of $|k|$.     
 
There are several recent papers         
about the approximate computation of highly oscillatory  
univariate integrals      
with the weight $\exp(2\pi\,i\,kx)$, where $x\in[0,1]$ and $k$  
is an integer (or $k \in \R$) which  is      
assumed to be large in the absolute sense, see      
Dom\'\i nguez, Graham and Smyshlyaev~\cite{DGS11},  
Iserles and N\o rsett~\cite{IN05},      
Melenk~\cite{Me10}, Chapter 3 of Olver~\cite{Ol08}, and      
Huybrechs and Olver~\cite{HO09} for a survey.      
Some authors mainly  
present asymptotic error bounds as $k$ goes to infinity      
for algorithms that use $n$ function or derivative values.  
It is usually done       
for $C^\infty$ or even analytic functions.        
There are not too many papers that contain explicit         
error bounds depending on $k$ and~$n$.  
Examples include~\cite{DGS11,Me10,Ol08}.     
All these papers also contain pointers to the further relevant literature. 
        
 
There is a discussion in the literature  
concerning the question whether ``high oscillation'', i.e.,          
large $|k|$, means that the problem is ``easy'' or ``difficult''.         
For this question it is useful  
to distinguish         
between the absolute and normalized error criteria and, in addition,  
it is important to know the initial errors.  
 
The absolute error      
criteria means that the error is at most $\e$, whereas the normalized       
error criteria means that the error is at most $\e$ times the initial      
error. The initial error is the error of the zero algorithm and      
only depends on the formulation of the problem. It turns out that in      
the setting of our paper the initial error is small for large $|k|$      
which makes the absolute error criterion easier than the      
normalized error criterion. We show that the answer to the question       
whether the problem is easy or difficult for large $|k|$ depends on      
the error criterion we choose as well as on the relation between      
$|k|$, $\e$ and the assumed smoothness of integrands.         
      
We did not find a computation of         
the initial error in the literature         
and we did not find lower bounds on the error of algorithms that use       
$n$ function or derivatives values. In this paper, we present the formulas      
for the initial error as well as matching lower and      
upper bounds on the minimal errors of      
algorithms.      
 
\section{Preliminaries}         
        
We study the Sobolev space $H^s$ for a finite $s\in\N$, i.e.,         
\begin{equation}         
H^s = \{ f: [0,1] \to \C \mid         
f^{(s-1)} \hbox{ is abs. cont., } f^{(s)} \in L_2 \}         
\end{equation}        
with the inner product        
\begin{equation} \label{eq:inner_product}      
\begin{split}      
\il f,g \ir_{s} \;&=\; \sum_{\ell=0}^{s-1} \int_0^1 f^{(\ell)}(x) \dx \;      
        \int_0^1 \widebar{g^{(\ell)}(x)} \dx \,+\, \int_0^1 f^{(s)}(x)\,      
\widebar{ g^{(s)}(x)}  \dx \\      
&=\; \sum_{\ell=0}^{s-1} \il f^{(\ell)},1\ir_0      
\widebar{\il g^{(\ell)},1\ir_0}       
                        \,+\, \il f^{(s)}, g^{(s)} \ir_0,      
\end{split}      
\end{equation}      
 where         
$\il f,g \ir_0 = \int_0^1 f(x)\, \widebar{g(x)} \, {\rm d} x$,      
and norm $\|f\|_{H^s}=\il f,f \ir_{s}^{1/2}$.     
We later comment on the space $H^\infty$ for $s=\infty$.         
     
\begin{rem}     
Probably the most standard inner product on the     
Sobolev space $H^s$ is     
\begin{equation}\label{newnorm1}     
\il f,g\ir_{s,*}=\sum_{\ell=0}^s\il f^{(\ell)},g^{(\ell)}\ir_0.     
\end{equation}     
Obviously, the norms $\|\cdot\|_{H^s}$ and    
$\|\cdot\|_{H^s_*}=\il f,f \ir_{s,*}^{1/2}$     
are equivalent. What is more surprising, the bounds on the    
embedding constants  are     
independent of $s$ and close to one. More precisely, we have      
\begin{equation}\label{equivalence}     
\tfrac{12}{13}\,\|f\|_{H^s_*}\le     
\|f\|_{H^s}\le \|f\|_{H^s_*}\ \ \ \ \ \mbox{for all}\ \ \     
f\in H^s\ \ \ \mbox{and} \ \ \ s\in\natural.      
\end{equation}     
The second inequality is trivial, whereas the first     
inequality seems to be new and its proof is given in the appendix.      
     
{}From~\eqref{equivalence} it clearly follows that all results     
presented in this paper for the     
space $H^s$ equipped with $\il \cdot,\cdot\ir_{s}$ are practically     
the same as for the space $H^s$ equipped with $\il     
\cdot,\cdot\ir_{s,*}$.      
We choose to work with the inner product $\il \cdot,\cdot\ir_s$ since      
the analysis in this case is easier and more straightforward.     
\qed      
\end{rem}     
\vskip 1pc        
We want to solve the following problem:      
        
\begin{itemize}         
\item        
What is the complexity of the approximate computation         
of oscillatory integrals of the form        
\begin{equation*}    
I_k (f) = \int_0^1 f(x) \,\eu^{-2\pi\,i\,kx} \, {\rm d} x,        
\ \ \ i=\sqrt{-1},        
\end{equation*} 
where $k \in \Z$ and $f \in H^s$?        
Our emphasize is on large $|k|$.        
We improve the known upper bounds and also prove matching lower bounds.         
\end{itemize}         
        
\noindent        
We ask (and answer) the same question     
also for the periodic case, i.e., for the subspace of~$H^s$ given by        
\begin{equation}         
\wt H^s = \{ f\in H^s \mid         
f^{(\ell)}(0) = f^{(\ell)}(1) \,\hbox{ for } \ell=0,1, \dots , s-1 \}        
\end{equation}         
equipped with the same inner product as for the space $H^s$.      
Note that for $f,g\in\wt H^s$ this inner product simplifies to       
\[\begin{split}      
\il f,g \ir_{s} \;=\; \il f,1\ir_0 \widebar{\il g,1\ir_0}       
                        \,+\, \il f^{(s)}, g^{(s)} \ir_0.      
\end{split}\]      
      
The results are presented in the following order.        
We first consider the integration problem      
for periodic functions, i.e.,~for         
functions from $\wt H^s$, and then,         
using this knowledge, we         
analyze the integration problem for the space $H^s$.     
        
The results of this paper could be stated also for         
real-valued functions, where $I_k(f)$ can be written, for example, as        
\[        
I_k (f) = \int_0^1 f(x) \cos (2\pi kx) \, {\rm d} x         
\]         
for $k\in\Z$, losing only some negligible constants.        
We decided to work with complex-valued functions to ease the notation.        
     
\section{The periodic case}   \label{sec:periodic}        
As already indicated,  we first analyze         
oscillatory integration over $\wt H^s$.        
That is, we want to approximate the integral         
\begin{equation}            
I_k (f) := \int_0^1 f(x) \,\eu^{-2\pi\,i\,kx} \, {\rm d} x=        
\int_0^1f(x)\,\cos(2\pi\,k\,x)\,{\rm d}x\ - \ i        
\int_0^1f(x)\,\sin(2\pi\,k\,x)\,{\rm d}x,        
\end{equation}         
where $k \in \Z$ and $f \in \wt H^s$         
with $s \in \nat$. Although $k$ can be any integer, the emphasis of        
this paper is for large $|k|$. In this case        
the weight functions        
$\cos(2\pi kx)$ and $\sin(2\pi kx)$ highly oscillate and therefore        
the approximation of $I_k$         
is called an \emph{(highly) oscillatory integration problem}.         
        
We consider the worst case error on the unit ball of $\wt H^s$         
for algorithms that use function values or, more generally,        
function and derivatives (up to order $s-1$) values.        
Note that for $f \in H^s$, the values        
$f^{(j)}(x)$        
are well defined for all $j=0,1,\dots,s-1$ and $x\in[0,1]$.       
       
It is well known that adaption does not help,  
see Bakhvalov \cite{B71},        
and linear algorithms are optimal,  
see Smolyak \cite{S65}. These results can        
be also found in e.g., \cite{No88,NW08,TWW88}.      
This means that without loss of        
generality we may consider linear algorithms of the form      
\begin{equation}    \label{algo}        
A_n(f) = \sum_{j=1}^n        
a_{j} f^{(\ell_j)} (x_j)         
\end{equation}        
for some  $a_{j} \in\C$, $\ell_j\in[0,s-1]$ and $x_j\in[0,1]$.        
Observe that we allow the use of  
derivatives $f^{(\ell_j)}(x_j)$ as in~\cite{IN05}.      
An important special case are linear algorithms        
that use only        
function values, i.e.,        
\begin{equation}    \label{algofun}        
A_n(f) = \sum_{j=1}^n a_{j} f (x_j)  .       
\end{equation}        
 
To summarize the problem: We want to compute  
a single Fourier coefficient $I_k(f)$ for  
$f \in H^s$ by algorithms of the form \eqref{algo} or \eqref{algofun}.  
The problem is described by $k$ and $s$, while $n$  
describes the amount of resources  
of an algorithm $A_n$. The algorithm (the knots $x_j$ and the weights $a_j$)  
may depend in an arbitrary way on $k$ and $s$ and $n$.  
For the upper bounds we work (in the periodic case) with equidistant knots.  
We do not know whether these knots are optimal but our lower bounds  
(valid for all algorithms) show that they are at least almost optimal.  
 
Of course, for $s=1$ there is no difference between~\eqref{algo}     
and~\eqref{algofun}.     
We will see that for all~$s$ 
the complexity results are similar for both       
classes of algorithms~\eqref{algo} and~\eqref{algofun}.       
       
The worst case error of $A_n$ is defined as       
$$        
\tilde e(A_n)= 
\sup_{f\in\wt H^s, \,\Vert f \Vert_{H^s} \le 1} |I_k(f) - A_n(f)|,       
$$        
whereas the $n$th minimal worst case error is         
\begin{equation}  \label{eq:error}       
\tilde e(n, k, s) := \inf_{A_n} \tilde e(A_n).        
\end{equation}         
We use the tilde to indicate that we consider the        
periodic case, i.e., the class $\wt H^s$.       
The particular case $n=0$ corresponds to the zero algorithm $A_0=0$        
and gives the so-called \emph{initial error}         
\begin{equation} \label{eq:initial}       
\tilde e(0, k, s) :=  
\sup_{f\in\wt H^s,\, \Vert f \Vert_{H^s} \le 1} |I_k(f)|        
= \Vert I_k \Vert_{\wt H^s\to\C} .         
\end{equation}         
       
\begin{rem}         
We believe that this is a simple but already interesting model problem        
for approximating highly oscillatory integrals.         
Later we plan to study the multivariate case and tractability         
and we believe that, from a practical point of view,         
the integrals         
$$        
S_k(f) = \int_{\reals^d} f(x)\,         
\eu^{ikx_1} \exp (-\Vert x \Vert_2^2) \, {\rm d} x         
$$        
for smooth integrands $f:  \R^d \to \C$         
are more interesting.         
      Here $x_1$ is the first coordinate of a vector $x \in \R^d$.  
We start, however, with the integral~\eqref{eq:problem} since         
it seems to be the simplest interesting        
case of oscillatory integrals.        
\qed        
\end{rem}         
        
We begin 
with the computation of the initial errors         
$\tilde e(0,k,s) = \Vert I_k \Vert_{\wt H^s\to\C}$.        
Since $I_k$ is a continuous linear        
functional defined on the Hilbert space $\wt H^s$, Riesz's theorem         
implies that for each $k\in\Z$ and $s\in\N$,         
there exists a function $\tilde h_{k,s} \in \wt H^s$ such that         
$$        
I_k(f) =       
\il f, \tilde h_{k,s} \ir_s  \quad \hbox{for all} \quad f \in \wt H^s.        
$$        
The function $\tilde h_{k,s}$ is called the representer of $I_k$         
for the space $\wt H^s$. It is well known and easy to show that         
$\Vert \tilde h_{k,s} \Vert_{H^s} = \Vert I_k\Vert_{\wt H^s\to\C}$.         
        
To find $\tilde h_{k,s}$ consider the particular function $e_k(x) =         
\eu^{2\pi\,i\, k x}$. Clearly, $e_k\in\wt H^s$.       
Using integration by parts, we obtain for $k\not=0$     
\begin{equation}\label{eq:repr}        
\begin{split}       
\il f,e_k \ir_{s}      
&= \il f,1\ir_0 \widebar{\il e_k,1\ir_0}       
                        \,+\, \il f^{(s)}, e_k^{(s)} \ir_0      
 = \il f^{(s)}, e_k^{(s)} \ir_0      
 = (-1)^s \il f, e_k^{(2s)} \ir_0  \\      
&= (2 \pi k)^{2s} \il f, e_k\ir_0       
= (2 \pi k)^{2s}\, I_k(f).       
\end{split}\end{equation}        
For $k=0$, we have $\il f,e_k \ir_{s}=\il f,1\ir_0=I_0(f)$.      
Hence we obtain         
        
\begin{prop}    \label{prop1}         
Let $k\neq0$. The representer of $I_k$ for the space $\wt H^s$  
is        
$$       
\tilde h_{k,s} (x) = (2\pi k)^{-2s}\, \eu^{2\pi\,i\, k x}        
$$       
and the initial error is         
$$       
\tilde e(0, k, s) =        
\Vert \tilde h_{k,s} \Vert_{H^s}  = \frac{1}{\bigl(2\pi |k| \bigr)^{s}}.        
$$       
Additionally, for $k=0$    
we have $\tilde h_{0,s} (x)=1$ and $\tilde e(0, 0, s) = 1$.      
\end{prop}         
       
We now present a few linear algorithms whose worst case errors are  
of order         
$(n+|k|)^{-s}$. We will prove later that this is the best possible        
order.        
       
For $n\ge1$, we first define the linear algorithm         
$$        
A_n^{\rm QMC}(f)=\frac1n\, 
\sum_{j=1}^nf\left(j/n\right)\,\eu^{-2\pi\,i\,k\,j/n}       
\qquad  \mbox{for all} \qquad f\in\wt H^s.        
$$        
We use the superscript ${\rm QMC}$         
to stress that the algorithm uses equal weights $1/n$         
for the function  $f(\cdot)\, \exp (-2\pi\,i\, k\,\cdot)$.         
This means that this is a QMC (quasi Monte Carlo) algorithm.         
This is the standard way to compute Fourier coefficients for periodic  
functions, with computation of integrals for a range of $k$  
facilitated by the FFT.  
Observe, however, that our problem is to compute $I_k(f)$ for a single $k$.  
 
As we shall see, the worst case error of $A_n^{\rm QMC}$         
is small only if $n$ is sufficiently large with respect to $|k|$.        
Later, we will modify the algorithm $A_n^{\rm QMC}$ to have a good error        
bound for all~$n$. First we prove the following theorem.         
       
\goodbreak     
     
\begin{thm} \label{thm:qmc}        
\qquad        
\begin{enumerate}[(i)]        
\item The worst case error of $A_n^{\rm QMC}$, $n\ge1$, is         
$$        
\tilde e(A_n^{\rm QMC})=         
\left(\sum_{j=1}^\infty\left(\frac1{\max\{1,(2\pi(jn+k))^{2s}\}}+       
\frac1{\max\{1,(2\pi(jn-k))^{2s}\}}\right)       
\right)^{1/2}.          
$$          
\item       
For any $1\le n\le |k|$ we have        
$$        
\tilde e(A_n^{\rm QMC}) >  \, \tilde e(0, k, s).         
$$        
\item        
For any $n> |k|$ we have        
$$        
\tilde e(A_n^{\rm QMC})=         
\left(\sum_{j=1}^\infty\left(\frac1{(2\pi(jn+k))^{2s}}+       
\frac1{(2\pi(jn-k))^{2s}}\right)       
\right)^{1/2}\le\frac{2}{(2\pi)^s} \,       
\,\frac1{(n-|k|)^{s}}.          
$$        
\item        
Let $\a\in(0,1)$. Then for $n> [(1+\a)/(1-\a)]\,|k|$ we have        
$$        
\tilde e(A_n^{\rm QMC}) <  \frac{2}{(2\pi\a)^s} \,       
\frac1{(n+|k|)^{s}}.        
$$        
\qed       
\end{enumerate}        
\end{thm}        
      
\begin{proof}       
For $h\in\Z$, let  $e_h(x)=\eu^{2\pi\,i\, hx}$ for $x\in[0,1]$.        
Since $f$ is a periodic function from~$\wt H^s$ we can write        
$$        
f(x)=\sum_{h\in\Z}\hat f_h\,e_h(x),        
$$        
with the Fourier coefficients        
$\hat f_h=\int_0^1f(x)\,\eu^{-2\pi\,i\,hx}\,{\rm d}x$.      
Since $f$ is smooth the last series is also pointwise convergent.     
Then         
$$        
A_n^{\rm QMC}(f)=\frac1n\,\sum_{h\in\Z}\hat f_h\,\sum_{j=1}^n\,        
\left[\eu^{2\pi\,i\,(h-k)/n}\right]^j.         
$$        
Note that the sum with respect to $j$ is zero if $h-k\not=0\mod n$,     
and is equal to $n$ if $h-k=0\mod n$. Therefore we can restrict $h$ to         
$h=k+jn$ with $j\in\Z$, and         
$$        
A_n^{\rm QMC}(f)=\sum_{j\in \Z}\hat f_{k+jn}.          
$$        
          We have $\hat f_k=I_k(f)$ which yields        
$$        
I_k(f)-A_n^{\rm QMC}(f)=-\sum_{j\in\Z\setminus 0}\hat f_{k+jn}.        
$$        
Let $a_h=\max\{1, (2\pi h)^{2s}\}$. Clearly $a_h=a_{-h}$.        
Since $e_h$'s are orthogonal in $\wt H^s$      
and $\|e_h\|^2_{H^s}=a_h$, for       
$f\in \wt H^s$ we have        
$$        
\|f\|^2_{H^s}=\sum_{h\in \Z}|\hat f_h|^2 a_h<\infty.        
$$         
Hence,        
$$        
|I_k(f)-A_n^{\rm QMC}(f)|        
=\left|\sum_{j\in\Z\setminus 0}\hat        
 f_{k+jn}\,a_{k+jn}^{1/2}\,a_{k+jn}^{-1/2}\right|\le        
\left(\sum_{j\in\Z\setminus0}|\hat f_{k+jn}|^2 a_{k+jn}\right)^{1/2}        
\left(\sum_{j\in\Z\setminus0} a_{k+jn}^{-1}\right)^{1/2}        
$$        
and        
$$        
|I_k(f)-A_n^{\rm QMC}(f)|\le \|f\|_{H^s}\,        
\left(\sum_{j=1}^\infty\left(\frac1{\max\{1,(2\pi(jn+k))^{2s}\}}+       
\frac1{\max\{1,(2\pi(jn-k))^{2s}\}}\right)       
\right)^{1/2}.         
$$        
The last inequality becomes an equality if we take         
$$        
f=c\,\sum_{j\in\Z\setminus0}\hat f_{k+jn}e_{k+jn}\ \ \ \         
\mbox{with}\ \ \ \ c\not=0, \ \ \mbox{and}\ \ \hat f_{k+jn}=a_{k+jn}^{-1}.        
$$         
We can choose $c$ such that $\|f\|_{H^s}=1$. This yields the formula for        
$\tilde e(A_n^{\rm QMC})$ and proves (i).        
\vskip 1pc        
Using (i), we obtain for $n\in[1,|k|]$ that        
\[       
\tilde e(A_n^{\rm QMC}) >        
\left(\frac1{\max\{1,(2\pi(|k|-n))^{2s}\}}\right)^{1/2}       
> \left(\frac1{\max\{1,(2\pi k)^{2s}\}}\right)^{1/2}        
= \tilde e(0,k,s).       
\]       
This proves (ii).       
\vskip 1pc       
We now estimate $\tilde e(A_n^{\rm QMC})$ for        
$n>|k|$. For such an        
integer $n$ and all $j\in\N$ we have       
\[\begin{split}        
\frac1{\max\{1,(2\pi(jn+k))^{2s}\}}+       
\frac1{\max\{1,(2\pi(jn-k))^{2s}\}}     
&\le\;\frac2{(2\pi(jn-|k|))^{2s}}.        
\end{split}\]        
This yields        
$$        
\tilde e^2(A_n^{\rm QMC})\,\le\,         
\frac2{(2\pi)^{2s}}\,\sum_{j=1}^\infty\frac1{(jn-|k|)^{2s}}.        
$$        
We have        
\begin{eqnarray*}        
\sum_{j=1}^\infty\frac1{(jn-|k|)^{2s}}&=&        
\frac1{(n-|k|)^{2s}}+        
\sum_{j=2}^\infty\frac1{(jn-|k|)^{2s}}\\        
&\le&        
\frac1{(n-|k|)^{2s}}+        
\int_1^\infty\frac{{\rm d}x}{(n x -|k|)^{2s}}\\        
&=&        
\frac1{(n-|k|)^{2s}}-\frac1{(2s-1)\,n}\left(nx-|k|\right)^{-(2s-1)}        
\bigg|_1^\infty\\        
&=&        
\frac1{(n-|k|)^{2s}}+\frac1{(2s-1)\,n\,(n-|k|)^{2s-1}}\\        
&=&        
\left(1+\frac{n-|k|}{(2s-1)\,n}\right)\,\frac1{(n-|k|)^{2s}}       
\,\le\,       
\frac{2s}{2s-1}\; \frac1{(n-|k|)^{2s}}\\       
&\le& \frac2{(n-|k|)^{2s}}.        
\end{eqnarray*}        
This completes the estimate of $\tilde e(A^{\rm QMC}_n)$ for $n>|k|$,       
and proves (iii).       
\vskip 1pc        
If $n> [(1+\a)/(1-\a)]\,|k|$ we have $n-|k| > \a(n+|k|)$       
and $(n-|k|)^{-s} <  \alpha^{-s}(n+|k|)^{-s}$.        
Then (iii)  easily yields (iv) and completes the proof.       
\end{proof}       
       
We comment on Theorem~\ref{thm:qmc}.        
Note that for $k=0$, the point (ii) cannot happen        
and the assumptions of (iii) and (iv) always hold.        
We now discuss this theorem for $k\not=0$. We start with (iv).        
Obviously, if $\alpha$ is close to zero then the condition         
on $n$ is relaxed. However, the upper bound on         
$\tilde e(A_n^{\rm QMC})$ is weaker 
since the factor $(2\pi\alpha)^{-s}$         
goes to infinity. On the other hand, if $\alpha$ goes to one then the        
condition on $n$ is more severe but the upper bound        
(in terms of $(n+|k|)^{-s}$) on        
$\tilde e(A_n^{\rm QMC})$ is better. This means that there is a tradeoff        
between the condition on $n$ and the quality of the upper bound on        
$\tilde e(A_n^{\rm QMC})$.         
This problem disappears if $n$ goes to infinity. Then         
we can take $\alpha$ close to one. In fact,         
the formula for $\tilde e(A_n^{\rm QMC})$ for $n$      
tending to infinity yields        
\begin{equation}\label{asympt}       
\lim_{n\to\infty}\tilde e(A_n^{\rm        
QMC})\,n^s=\frac{(2\zeta(2s))^{1/2}}{(2\pi)^s},        
\end{equation}       
where $\zeta$ is the Riemann zeta function,        
$\zeta(x)=\sum_{j=1}^\infty j^{-x}$ for $x>1$.         
       
\begin{rem}       
It is interesting that the right hand side of~\eqref{asympt}       
appears for other problems.       
First of all, it is        
the norm of the embedding of $\wt H^s\cap\{f\in \wt H^s| \       
\int_0^1f(x)\,{\rm d}x=0\}$,        
equipped with the norm        
$\|f^{(s)}\|_{L_2}$, into $C([0,1])$, see \cite[Theorem 1.2]{KYWNT06}.        
This was proven by calculating the diagonal values of the corresponding        
reproducing kernel.       
Moreover, the right hand side of~\eqref{asympt}       
equals $\frac{1}{s!} \|B_s\|_{L_2}$, where        
$B_s$ is the Bernoulli polynomial of degree $s$,       
see \cite[Lemma~2.16]{KYWNT06}.       
\qed       
\end{rem}       
       
We now comment on Theorem~\ref{thm:qmc} when $n\in[1,|k|]$.       
In this case we know that the algorithm $A_n^{\rm QMC}$ is even worse        
than the zero algorithm. For instance, take $n=|k|$. Then it is easy       
to conclude from (i) that       
\begin{equation}  \label{bound33}        
\tilde e(A_n^{\rm QMC})^2 > \, 1 + \frac{1}{(4\pi k )^{2s}}.       
\end{equation}        
Note that \eqref{bound33} is almost worst possible       
since every quadrature rule $A_n$        
with positive weights which sum up to one satisfies        
\begin{eqnarray*}       
\tilde e(A_n)        
&\le& \sup_{f\in\wt H^s, \,\Vert f \Vert_{H^s} \le 1} \left(|I_k(f)| +       
  |A_n(f)|\right)\le        
 \Vert I_k \Vert_{\wt H^s\to\C} +\sup_{f\in\wt H^s, \,\Vert f   
   \Vert_{H^s}   
   \le 1}\,\sup_{x\in[0,1]}|f(x)|\\       
&= & \frac{1}{(2\pi |k| )^{s}}+       
\| {\rm Id}\|.        
\end{eqnarray*}       
Here, ${\rm Id}:\wt H^s\to C([0,1])$ is the embedding operator, i.e.,     
${\rm Id}f=f$ for all $f\in H^s$. We can estimate its norm as follows.     
We know that $\wt H^s$ is a reproducing kernel Hilbert space with the      
kernel\footnote{The formula of the reproducing kernel of $H^s$      
given as      
(10.2.4) on page 130 in \cite{W90} and as Example 21 on page 320 in      
\cite{BT04} has a typo. The term $B^*_s(x)B^*_s(t)$ should be replaced     
by $\sum_{j=1}^sB^*_j(x)B^*_j(t)$, as is correctly stated in the original      
paper~\cite{CW79}, where this result is proved.}      
\[        
\wt K_s(x,t) \;=\; 1+(-1)^{s-1}\,B^*_{2s}(\{x-t\})     
\;=\; 1+\sum_{h\in\Z\setminus0}\,         
\frac{\eu^{2\pi\,i\, h(x-t)}}{(2\pi h)^{2s}}     
\]         
where $B^*_k=B_k/k!$ is the $k$th normalized Bernoulli polynomial,      
see \eqref{eq:bernoulli}, and      
$\{x-t\}$ is the fractional part of $x-t$.      
This implies for $f$ with $\|f\|_{H^s}\le1$ that     
$$     
f^2(x)=\il f,\wt K_s(\cdot,x)\ir_s^2\le \|f\|_{H^s}^2\,\wt K_s(x,x)\le     
1+\frac2{(2\pi)^{2s}}\,\sum_{j=1}^\infty\frac1{j^{2s}}.     
$$     
Hence, $\|{\rm Id}\|^2\le 1+2\zeta(2s)/(2\pi)^{2s}$ and      
$$     
\tilde e(A_n)\le 1  
+ \frac{1}{(2\pi |k| )^{s}}+  \frac{(2\zeta(2s))^{1/2}}{(2\pi)^{s}}     
$$     
which for large $s$ is close to one as the right hand-side of~\eqref{bound33}. 
 
\vskip 1pc       
We now show how to modify the algorithm $A_n^{\rm QMC}$ such that its        
worst case error is smaller than the initial error $\tilde e(0,k,s)$        
with no condition on $n$. It turns out that the weight $n^{-1}$       
used by the algorithm $A_n^{\rm QMC}$ is too large.        
\begin{thm}\label{optya}       
For $a\in\R$, consider the algorithm of the form       
$$       
A_{n,a}(f)=\frac{a}{n}\,\sum_{j=1}^nf(j/n)\,\exp^{-2\pi\,i\,k\,j/n}       
\ \ \ \ \ \mbox{for all}\ \ \ \ \ f\in \tilde h^s.       
$$       
The worst case error of $A_{n,a}$ is minimized with respect to $a$      
for       
$$       
a=a^*_n=\frac{[\tilde e(0,k,s)]^2}{[\tilde e(0,k,s)]^2+[\tilde e(A_n^{\rm       
    QMC})]^2},       
$$       
and        
$$       
\tilde e(A_{n,a^*_n})= \frac{\tilde e(0,k,s)\ \tilde e(A_n^{\rm QMC})}       
{\sqrt{[\tilde e(0,k,s)]^2+[\tilde e(A_n^{\rm QMC})]^2}}.       
$$       
Clearly,        
$$       
a^*_n<1\ \ \ \ \ \mbox{and}\ \ \ \ \        
\tilde e(A_{n,a^*_n})<\min\{\tilde e(0,k,s),       
\tilde e(A_n^{\rm QMC})\}.       
$$       
\qed      
\end{thm}       
\begin{proof}       
Repeating the analysis of the first part of the proof of       
Theorem~\ref{thm:qmc}, we obtain        
$$       
I_k(f)-A_{n,a}(f)=(1-a)\hat f_k -a\sum_{j\in\Z\setminus 0}\hat f_{k+jn}.       
$$       
Similarly as before we use $a_h=\max\{1, (2\pi h)^{2s}\}$ and       
conclude that        
$$       
\tilde e(A_{n,a})=\left(\frac{(1-a)^2}{a_k}+a^2\,\       
\sum_{j\in\Z\setminus 0}\frac1{a_{k+jn}}\right)^{1/2}=       
\left((1-a)^2[\tilde e(0,k,s)]^2+a^2[\tilde e(A_n^{\rm QMC})]^2\right)^{1/2}.       
$$       
Clearly, the last expression is minimized with respect to $a$ for       
$a=a^*_n$ from which we obtain the form of $\tilde e(A_{n,a^*_n})$.       
This completes the proof.       
\end{proof}       
       
We discuss $a^*_n$ which decreases the weight $n^{-1}$ in the       
algorithm $A_{n,a^*_n}$. For $n\in[1,|k|]$, the point (ii) of        
Theorem~\ref{thm:qmc} yields that $a^*_n<1/2$. For $n=|k|\ge1$ we know       
from~\eqref{bound33} that $\tilde e(A_n^{\rm QMC})>1$, and therefore       
$a^*_n\le [\tilde e(0,k,s)]^2= (2\pi|k|)^{-2s}$ which is       
polynomially small in $|k|$ and exponentially small in $s$.       
On the other hand, if $k$ is fixed      
and $n$ goes to infinity then $a^*_n$ goes to one        
and the algorithm $A_{n,a^*_n}$ becomes the same as the algorithm        
$A_n^{\rm QMC}$.        
         
The algorithm $A_{n,a^*_n}$ has a (small) computational drawback       
since it requires the exact value of $a^*_n$ which is given by the       
infinite series describing the worst case error of $\tilde e(A_n^{\rm       
  QMC})$. Of course, it can be precomputed to an arbitrary accuracy.       
       
There is another simple idea how to modify the algorithm       
$A_n^{\rm QMC}$ without computing $a^*_n$.        
Namely, for small $n$ we  use the zero algorithm whereas for large $n$         
we use the algorithm $A_n^{\rm QMC}$. More precisely,        
for $n=0,1,\dots$, we define        
the algorithm        
\begin{equation} \label{eq:A_per}       
A^*_n(f)=\begin{cases}        
\ \ \ 0& \ \ \mbox{if} \ \ n=0\ \ \mbox{or}\ \ n<2|k|,\\        
\ \ \ A_n^{\rm QMC}(f)&\ \ \mbox{if}  \ \ n\ge\max(1,2|k|).        
\end{cases}        
\end{equation}        
The algorithm $A_n^*$ uses no information on $f$ if $n=0$ or $n<2|k|$,       
and $n$ function values otherwise.       
Based on Theorem~\ref{thm:qmc} and the discussion after its proof         
it is easy to show        
       
\begin{cor}\label{coranstar}        
We have        
$$        
\tilde e(A^*_n)\,        
\begin{cases}        
\ \ =1&\ \ \ \mbox{for}\ \ \ k=0\ \  \mbox{and}\ \ n=0,\\        
\ \ =\frac1{(2\pi|k|)^s}&\ \ \ \mbox{for}\ \ \ k\not=0\ \ \mbox{and}\ \        
n\in[0,2|k|),\\        
\ \ \le \frac{2}{(2\pi)^s}\frac1{(n-|k|)^s}&\ \ \ \mbox{for}\ \ \         
n\ge\max(1,2|k|).         
\end{cases}        
$$        
Furthermore,       
\begin{equation}\label{estanstar}        
\tilde e(A^*_n) \;\le\;        
\left(\frac3{2\pi}\right)^s\,\frac{2}{(n+|k|)^s}\ \ \ \         
\mbox{for all}\ \ \ \ n\ge1.       
\end{equation}       
\qed       
\end{cor}         
       
\begin{proof}       
Assume first that $k=0$.        
Then for $n=0$ we have $A^*_0=0$ and $\tilde e(A^*_0)=1$.        
For $n\ge1$ we have $A^*_n=A_n^{\rm QMC}$ and we use       
Theorem~\ref{thm:qmc}(iii) to get the third estimate on $\tilde e(A^*_n)$.     
       
Assume now that $k\not=0$. For $n\in[0,2|k|)$  the error of $A^*_0=0$  
is the       
initial error which is $(2\pi|k|)^{-s}$. For $n\ge2|k|$       
we have $A^*_n=A_n^{\rm QMC}$ and the estimate on       
$\tilde e(A^*_n)$ follows from Theorem~\ref{thm:qmc}(iii).        
       
We now prove the estimate~\eqref{estanstar}. Again assume first that       
$k=0$.           
Consider first the case $n\ge \max(1,2|k|)$. We can now apply       
Theorem~\ref{thm:qmc}(iv) with $\alpha=1/3$ and then       
$$       
\tilde e(A_n^*)=\tilde e  (A_n^{\rm       
  QMC})\le\left(\frac3{2\pi}\right)^s\,\frac2{(n+|k|)^s},       
$$       
as claimed. It remains to consider the case $n\in[1,2|k|)$ for       
$k\not=0$. Then $|k|>(n+|k|)/3$ and         
$$        
\tilde e(A^*_n)
=\frac1{(2\pi|k|)^s}\le \left(\frac3{2\pi(n+|k|)}\right)^s=        
\left(\frac3{2\pi}\right)^s\,\frac1{(n+|k|)^s},        
$$        
as claimed. This completes the proof.        
\end{proof}       
     
\begin{rem}                      
Another possible modification of the 
algorithm $A_n^{\rm QMC}$ for small $n$ is  
the ``Filon-type'' approach, see e.g.~\cite{DGS11,HO09,IN05,Me10}. 
For such an algorithm one assumes to have  
a given set of functions $\{b_1,\dots,b_N\}$,  
e.g.~$b_\ell$ could be polynomials, and that we know  
the values $I_k(b_\ell)$ for all $\ell=1,\dots,N$.  
Using these functions and the function values $y_j=f(x_j)$, $j=1,\dots,n$,  
we compute an approximation of the input $f$  
of the form $F_n(x)=\sum_{\ell=1}^N c_\ell(y_1,\dots,y_n)\,b_\ell(x)$.  
The ``Filon-type'' algorithm $A_n^{**}$ for $I_k$ is then given by  
$A_n^{**}(f)=\sum_{\ell=1}^N c_\ell(y_1,\dots,y_n)\,I_k(b_\ell)$.  
In Theorem \ref{thm:int1}  we prove a lower bound valid for all 
algorithms that use function and derivatives values. This lower bound 
shows that our algorithm $A_n^*$ is almost optimal but does not 
exclude the possibility that a suitably chosen Filon-type algorithm  
is even slightly better than $A_n^*$, i.e., $\wt e(A^{**}_n)<\wt 
e(A^*_n)$ although $\wt e(A^{**}_n)$ must be also of order  
$(n+|k|)^{-s}$. 
\end{rem} 

We stress that all algorithms considered so far use only function       
values although we allow also computation of derivatives up to order       
$s-1$. Furthermore, they use function values at equally spaced       
points and use the same weights $n^{-1}$ or $a^*_nn^{-1}$ for large       
$n$. Although algorithms that minimize the worst case error are       
probably not of this form, we now prove a lower bound on        
the order of convergence of an arbitrary algorithm, and show       
that this order is       
$(n+|k|)^s$. Hence the algorithm $A^*_n$ enjoys the best possible       
order of convergence. Additionally, the algorithm $A^*_n$ is easy to       
implement.          
       
\begin{thm} \label{thm:int1}        
Consider the integration problem $I_k$ defined over the        
space $\wt H^s$ of periodic functions with $s \in \nat$.        
Let $\tilde e(n,k,s)$ be the $n$th minimal worst case error       
of all algorithms that use at most $n$ function or derivatives (up to       
order $s-1$) values, see~\eqref{eq:error}.       
There is a number $c_s >0$ such that         
$$        
\frac{c_s}{(n+|k|)^s} \;\le\; \tilde e(n,k,s) \;\le\;        
\left(\frac3{2\pi}\right)^s\,\frac{2}{(n+|k|)^s}        
$$        
for all $k\in\Z$ and $n\in\N$.         
\qed      
\end{thm}         
        
\begin{proof}         
The upper bound has been already shown for the algorithm $A^*_n$.        
Hence, we only need to prove the lower bound.        
       
Let $A_n$ be an arbitrary algorithm of the form \eqref{algo} that uses       
$f^{(\ell_j)}(x_j)$ for some $\ell_j\in[0,s-1]$ and $x_j\in[0,1]$       
for $j=1,2,\dots,n$.       
Suppose that for $f\in \wt H^s$ we get        
$f^{(\ell_j)}(x_j)=0$ for all $j=1,2,\dots,n$.        
Since $-f$ also belongs to $\wt H^s$,       
the algorithm $A_n$ cannot distinguish between $I_k(f)$ and       
$I_k(-f)=-I_k(f)$. Therefore $|I_k(f)|$ is a lower bound on the worst       
case error of~$A_n$. This leads to a well-known inequality         
$$       
\tilde  e (A_n) \ge \sup \{        
|I_k(f)| \; : \;  f \in \wt H^s, \,        
\Vert f \Vert_{H^s} \le 1, \, N(f) = 0 \},        
$$       
where       
\[      
N(f) = [f^{(\ell_j)} (x_j) , \;  j= 1, 2,\dots ,  n]      
\]      
Below we will construct a function $f$ with large $|I_k(f)|$ and       
all of the $s\cdot n$ values       
$f^{(\ell)} (x_j),   j=1,\dots, n,\, \ell= 0,\dots , s-1$, are      
equal to zero.       
Obviously, such a function $f$ satisfies $N(f)=0$.      
      
We consider a real-valued $f$ and the real part of $I_k(f)$ which is  
$$       
\widehat I_k (f) = \int_0^1 f(x) \cos (2 \pi k x) \, {\rm d} x       
$$       
of $I_k(f)$.        
Define the disjoint        
subintervals $T_{i,k} \subset [0,1]$ such that for all $x \in T_{i,k}$ we have  
$|\cos (2\pi k x)| \ge 1/\sqrt{2}$.        
There are $2|k|+1$ such subintervals. For $k=0$  we have $T_{1,0}=[0,1]$,
whereas for $k\not=0$ the lengths of~$T_{i,k}$'s for       
$i=1, 2, \dots, 2|k|+1$ are         
$\frac{1}{8|k|}, \frac{1}{4|k|}, \dots, \frac{1}{4|k|}, \frac{1}{8|k|}$       
with the total length $1/2$. The points        
$x_1, \dots , x_n$ used by $A_n$ may divide the        
$T_{i,k}$ further and altogether we obtain $m \in[2|k|+1, 2|k|+1+n]$        
intervals $\widehat T_{1,k} \dots , \widehat T_{m,k}$;        
all the endpoints of the $\widehat T_{i,k}$ coincide with an endpoint        
of one of the $T_{i,k}$ or are one of the $x_j$.        
Again, the sum of the lengths of the $\widehat T_{i,k}$ is~$1/2$.        
       
We define        
$\Phi (x) = d_s ( \cos^2 (\pi x/2))^s$ for $|x| \le 1$ and        
$\Phi(x)=0$ otherwise. Then $\Phi \in C^s(\R)$ and we can choose        
$d_s>0$ in such a way that        
$\Vert \Phi\Vert_{H^s([-1,1])} = 1$.        
       
Let the length of the interval        
$\widehat T_{i,k}$ be $1/n_i$ and let $y_i$ be its midpoint.       
For $i=1,2, \dots , m$, we define a scaled version of $\Phi$ by       
$$       
\Phi_i(x) =        
\frac{\sgn(\cos(2\pi k y_i))}{(2 n_i)^s} \,        
\Phi (2n_i x-2n_i y_i)\ \ \ \ \ \mbox{for all}       
\ \ \ \ \ x\in\R.       
$$       
Note that the support of $\Phi_i$ is $\widehat T_{i,k}$ and        
$\|\Phi_i\|_{H^s([-1,1])}\le1$. Furthermore,       
\begin{eqnarray*}       
\widehat I_k(\Phi_i)&=&       
\frac1{(2n_i)^s}\,       
\int_{\widehat T_{i,k}}|\cos(2\pi k x)|\Phi(2n_i(x-y_i))\,{\rm d}x       
\ge        
\frac1{2^{s+1/2}\,n_i^s}\,       
\int_{\widehat T_{i,k}}\Phi(2n_i(x-y_i))\,{\rm d}x\\       
&=&       
\frac{d_s}{2^{s+3/2}\,n_i^{s+1}}\,       
\int_{-1}^{1}(\cos^2(\pi t/2))^s\,{\rm d}t.       
\end{eqnarray*}       
       
Finally we define our ``fooling function''        
by        
$$       
f= \sum_{i=1}^m \Phi_i.       
$$        
It is easy to check that        
$f \in \wt H^s$ with $N(f)=0$ and $\Vert f \Vert_{H^s} \le 1$.        
We can also estimate the integral and obtain        
$$       
|I_k(f)| \ge |\widehat I_k(f)|=       
\sum_{i=1}^m\widehat I_k(\Phi_i)\ge       
\tilde c_s \sum_{i=1}^m n_i^{-s-1}       
$$       
with       
$$       
\tilde c_s=\frac{d_s}{2^{s+3/2}}\,\int_{-1}^{1}(\cos^2(\pi t/2))^s\, 
{\rm d}t>0.       
$$       
It is easy to check by standard means that       
$$       
\min_{n_i:\       
  \sum_{i=1}^mn_i^{-1}=1/2}\,      
\sum_{i=1}^mn_i^{-s-1}=\frac1{2^{s+1}}\,\frac1{m^s}       
\ge \frac1{2^{s+1}\,(2|k|+1+n)^s}\ge       
\frac1{2\cdot 4^s(n+|k|)^s}.       
$$       
This proves the lower bound with $c_s=\tilde c_s/(2\cdot 4^s)$.        
\end{proof}        
We stress that the lower bound in Theorem~\ref{thm:int1} holds       
for a larger class of algorithms than the class~\eqref{algo}      
for $s>1$. Namely it holds for algorithms       
$$      
A_n(f)=\sum_{j=1}^n\,\sum_{\ell=1}^{s-1}a_{j,\ell}\,f^{(\ell)}(x_j)      
$$      
for arbitrary $a_{j,\ell}\in\C$ and $x_j\in[0,1]$. That is, we now use      
$n\cdot s$ values of $f$ and its derivatives       
instead of $n$, however, we still have ``only'' $n$ sample points to      
choose.       
     
Theorem~\ref{thm:int1}         
states that both lower and upper bounds on the        
$n$th minimal error decay with~$|k|$. Does it really mean that high        
oscillation makes the problem easy? The answer to this question        
depends on whether we consider the absolute or normalized error        
criterion.         
       
For the \emph{absolute error criterion}, the information        
complexity~$\tilde n^{\rm abs}(\e,k,s)$ is defined as the minimal         
$n$ for which the error is at most $\e\in(0,1)$. That is,        
$$        
\tilde n^{\rm abs}(\e,k,s)        
=\min\left\{\,n\ |\ \ \tilde e(n,k,s)\le \e\,\right\}.        
$$        
Clearly, $\tilde n^{\rm abs}(\e,k,s)=0$         
for $\e\ge \tilde e(0,k,s)$ since we can solve the problem by the zero        
algorithm. For $\e< \tilde e(0,k,s)$        
we can bound $\tilde n^{\rm abs}(\e,k,s)$  by Theorem~\ref{thm:int1}.        
This implies the following corollary.        
       
\begin{cor}  \label{coro:per_abs_n}      
Consider the absolute error criterion for the integration problem       
$I_k$ defined over the periodic space $\wt H^s$.        
Let $c_s$ be from Theorem~\ref{thm:int1}.       
\begin{itemize}        
\item For $k=0$ and all $\e\in(0,1)$ we have         
$$        
c_s^{1/s}\,\left(\frac1{\e}\right)^{1/s}        
\le \tilde n^{\rm abs}(\e,0,s)\le         
\left\lceil        
  \frac3{2\pi}\left(\frac{\sqrt{2}}{\e}\right)^{1/s}\right\rceil.        
$$        
\item For $k\not=0$ and $\e\in[1/(2\pi|k|)^s,1)$ we have        
$$        
\tilde n^{\rm abs}(\e,k,s)=0,        
$$        
whereas for $\e\in(0,1/(2\pi|k|)^s)$ we have        
$$        
c_s^{1/s}\,\left(\frac1{\e}\right)^{1/s}\,-\,|k|        
\le \tilde n^{\rm abs}(\e,k,s)\le         
\left\lceil        
  \frac3{2\pi}\left(\frac{\sqrt{2}}{\e}\right)^{1/s}\right\rceil\,-\, |k|.        
$$        
\end{itemize}        
\qed       
\end{cor}        
       
This means that for the absolute error criterion the problem becomes easier         
for large $|k|$,        
but the asymptotic behavior of $\tilde n^{\rm abs}(\e,k,s)$, as $\e\to0$,        
does not depend on $k$.       
       
We now turn to the \emph{normalized error criterion} in which we want        
to reduce the initial error $\tilde e(0,k,s)$ by a factor $\e\in(0,1)$. That        
is, the information complexity  $\tilde n^{\rm nor}(\e,k,s)$ is defined as         
$$        
\tilde n^{\rm nor}(\e,k,s)        
=\min\left\{\,n\ |\ \ \tilde e(n,k,s)\le \e\,\tilde e(0,k,s)\,\right\}.        
$$        
In this case we always have $\tilde n^{\rm nor}(\e,k,s)\ge1$.        
Note that for $k=0$ we have $\tilde e(0,0,s)=1$ and there is no difference        
between the normalized and absolute error criteria.        
        
For $k\not=0$ the situation is quite different. From        
Theorem~\ref{thm:qmc}, Theorem~\ref{thm:int1} and Proposition~\ref{prop1}        
it is easy to  prove the following corollary.        
       
\begin{cor}        
Consider the normalized error criterion for the integration problem       
$I_k$ defined over the periodic space $\wt H^s$.        
Let $c_s$ be from Theorem~\ref{thm:int1}.       
       
For all $k\not=0$ and all $\e\in(0,1)$ we have        
$$        
|k|\, \left(2\pi \left(\frac{c_s}{\e}       
\right)^{1/s}-1 \right)       
\le \tilde n^{\rm nor}(\e,k,s) \le \;       
|k|\, \left\lceil 3\left(\frac{\sqrt{2}}{\e}\right)^{1/s}-1\right\rceil,       
$$        
which can be written as       
$$       
\tilde n^{\rm nor}(\e,k,s)      
=\Theta\left(\frac{|k|}{\e^{1/s}}\right)       
\ \ \ \ \ \mbox{as}\ \ \ \ \ \ \e\to0.       
$$       
\qed       
\end{cor}        
        
Hence, for the normalized error criterion the problem becomes harder for        
large $|k|$. It is interesting that the dependence        
on $|k|$ is linear and does not depend on~$s$.        
In particular, for fixed $s$ and fixed $\e<(2\pi)^s c_s$ we have      
\begin{equation}\label{pernor}      
\lim_{|k|\to\infty}\tilde n^{\rm nor}(\e,k,s)=\infty.       
\end{equation}      
     
\section{The non-periodic case}   \label{sec:non-periodic}        
        
We now turn to the case of non-periodic functions, i.e.,~we         
consider the Sobolev space        
\begin{equation}         
H^s = \{ f: [0,1] \to \C \mid         
f^{(s-1)} \hbox{ is abs. cont., } f^{(s)} \in L_2 \}      
\end{equation}        
for a finite $s\in\N$. The inner product $\il\cdot,\cdot\ir_s$ in $H^s$         
is again defined by \eqref{eq:inner_product}.      
     
Clearly, for all $j=1,2,\dots,s$ we      
have $H^s\subset H^j$ and $\|f\|_{H^j}\le \|f\|_{H^s}$      
for all $f\in H^s$.      
This follows from the inequality     
\[     
\int_0^1|f'(x)|^2 \dx \;\ge\; \int_0^1|f(x)|^2 \dx \,-\,     
\left(\int_0^1 f(x) \dx \right)^2     
\]     
for differentiable functions $f$     
and implies that the unit ball of $H^s$ is a subset of the unit ball      
of $H^j$.        
     
Again we want to approximate the integral         
\begin{equation}             
I_k (f) := \int_0^1 f(x) \,\eu^{-2\pi\,i\,kx} \, {\rm d} x,        
\end{equation}         
where $k\in\Z$ and $f \in H^s$         
with $s \in \nat$. Without loss of generality we consider       
linear algorithms~$A_n$ of the form~\eqref{algo}.       
Similarly as before, we define the worst case error of $A_n$ as       
\[       
e(A_n) \,:=\, \sup_{f\in H^s, \,\Vert f \Vert_{H^s} \le 1} |I_k(f) -       
A_n(f)|,     
\]       
and the $n$th minimal worst case error as       
\[       
e(n, k, s) \,:=\, \inf_{A_n}\, e(A_n).       
\]       
In particular, the initial error is given by        
\[       
e(0, k, s) := \sup_{f\in H^s, \Vert f \Vert_{H^s} \le 1} |I_k(f)|        
= \Vert I_k \Vert_{H^s\to\C},         
\]       
compare with~\eqref{eq:error} and \eqref{eq:initial}. We do not now       
use the tilde to stress        
the non-periodic case.        
       
Note that $H^s$ is obviously a superset of $\wt H^s$ and hence,        
lower bounds that were proved in Section~\ref{sec:periodic}        
for $\wt H^s$ also hold for $H^s$, i.e.,       
\begin{equation}\label{lownonper}       
e(n,k,s)\ge\tilde e(n,k,s).       
\end{equation}       
We start with the        
computation of the initial error.       
As we shall see, for large $s$ and $|k|$,       
it is now much larger than for the periodic case.       
In particular, the initial error for $k\neq0$ does not tend to        
zero if $s$ tends to infinity.       
       
Similarly to \eqref{eq:repr} we want to compute the representer $h_{k,s}$       
of $I_k$ in $H^s$. Using the same functions $e_k(x) = \eu^{2\pi i k x}$,       
which satisfy $\|e_k\|_{H^s}=(2\pi|k|)^s$,       
we obtain       
\begin{equation}\label{eq:repr2}      
\begin{split}      
\il f,e_k \ir_{s} &= \il f^{(s)}, e_k^{(s)} \ir_0      
 = (-1)^s \il f, e_k^{(2s)} \ir_0        
                +      
\sum_{\ell=0}^{s-1} (-1)^\ell      
\Bigl[f^{(s-\ell-1)} \widebar{e_k^{(s+\ell)}}\Bigr]_0^1\\     
&= (2\pi k)^{2s}\, I_k(f)        
                + (-1)^s\sum_{\ell=1}^{s}(2\pi i k)^{2s-\ell}\,     
                        \Bigl(f^{(\ell-1)}(1)-f^{(\ell-1)}(0)\Bigr).       
\end{split}\end{equation}      
Here we use the fact that $k$ is an integer. 
Surprisingly, the functionals $f^{(\ell-1)}(1)-f^{(\ell-1)}(0)$, 
$\ell=1,\dots,s$,      
or more precisely their representers in $H^s$, have some nice properties      
that will be useful in the following analysis.      
These representers are given by the normalized Bernoulli polynomials       
\begin{equation}\label{eq:bernoulli}     
B^*_\ell(x) \;=\; \frac{1}{\ell!}\, B_\ell(x),      
\end{equation}     
where the Bernoulli polynomials $B_\ell$, $\ell\ge0$, are the       
unique polynomials that are given by       
\[      
\int_t^{t+1} B_\ell(x) \dx \;=\; t^\ell      
\ \ \ \ \mbox{for all} \ \ \ t\in\reals\ \ \mbox{with}\ \ 0^0=1.     
\]      
To see this,  note that $[B^*_\ell]' =B^*_{\ell-1}$ as well as      
$\int_0^1 B^*_\ell(x)\dx=0$, $\ell\ge1$, and       
$\int_0^1 B^*_0(x)\dx=1$.      
In particular, this implies for $f\in H^s$ and $\ell\le s$ that      
\[      
\il f,B^*_\ell\ir_{s} =      
\il f^{(\ell)},1\ir_0 = f^{(\ell-1)}(1) - f^{(\ell-1)}(0),      
\]      
which proves the claim.     
Additionally, this shows      
\begin{equation}\label{eq:bernoulli2}     
\|B^*_\ell\|_{H^s} =1 \quad\text{ and }\quad \il B^*_\ell, B^*_m\ir_{s}=0      
\end{equation}     
for $\ell,m\in\{0,1,\dots,s\}$ with $\ell\neq m$ and, consequently,     
\[      
H^s \;=\; \wt H^s \oplus \{B^*_1\} \oplus  \dots \oplus \{B^*_s\},      
\]      
see e.g.~\cite[Section~10.2]{W90}.     
     
Using \eqref{eq:repr2} we obtain the following proposition.      
      
\begin{prop}    \label{prop:non_initial}        
Let $k\neq0$. The representer of $I_k$ in $H^s$ is        
$$       
h_{k,s} (x) = (2\pi k)^{-2s}\, \eu^{2\pi\,i\, k x}        
\,-\, \sum_{\ell=1}^{s}(-1)^\ell      
(2\pi i k)^{-\ell}\, B^*_\ell(x)      
$$       
and the initial error is         
$$       
e(0, k, s) =        
\Vert h_{k,s} \Vert_{H^s}  =\;       
\sqrt{\frac{2}{\bigl(2\pi k \bigr)^{2s}} +       
        \sum_{\ell=1}^{s-1}\frac{1}{\bigl(2\pi k \bigr)^{2\ell}}}     
\;=\; \frac{\beta_{k,s}}{2\pi |k|}      
$$       
with $\beta_{k,1}=\sqrt{2}$ and      
\[     
1 \,\le\, \beta_{k,s}      
\,=\, \sqrt{\frac{(2\pi k)^{2s}+(2\pi k)^{2}-2}{(2\pi k)^{2s}-(2\pi k)^{2(s-1)}}}     
\,\le\, \sqrt{1+ \frac{2}{(2\pi k)^2-1}}\,\le 1.02566     
\]     
for $s>1$. Note that $\lim_{k\to\infty} \beta_{k,s} =1$.     
     
For $k=0$, the representer is $h_{0,s}=1$ and     
the initial error is one, $\tilde e(0,0,s)=1$.        
\end{prop}         
        
\vskip 1pc     
We are ready to discuss algorithms for the non-periodic case.     
One of the ideas to get such algorithms is first to periodize functions     
$f$ from $H^s$ by computing      
$f^{(0)}(0), \dots , f^{(s-1)}(0)$ and        
$f^{(0)}(1), \dots , f^{(s-1)}(1)$, and then apply the algorithm      
$A^*_{n-2s}$ from Section~\ref{sec:periodic}. Of course, this requires     
to assume that $n\ge 2s$ which is a bad assumption if $s$ is      
large or even impossible     
to satisfy if $s=\infty$. Therefore for $n<2s$ we need to proceed differently.      
As already discussed $f\in H^s$ implies that $f\in H^j$ for all $j\le s$.      
Therefore      
we can use periodization for $H^j$ by computing      
$f^{(0)}(0), \dots , f^{(j-1)}(0)$ and        
$f^{(0)}(1), \dots , f^{(j-1)}(1)$ as long as $n\ge2j$.      
Then we can again apply the algorithm $A^*_{n-2j}$   from      
Section~\ref{sec:periodic}.      
Formally, this algorithm was studied      
only for $H^s$ but it is obvious that its error     
can be also analyzed for $H^j$ with the change of $s$ to $j$.     
     
Another idea to obtain algorithms for small $n$ relative to $s$      
is to use the integration of Taylor's expansion      
of $f\in H^s$ at $\tfrac12$. As we shall see this approach is      
appropriate if $|k|$ is relatively small with respect to $n$.       
To explain these ideas more precisely we need some preparations.

\subsection{Periodization}\label{subsec:periodization}     
     
Let $1\le j\le s$ be given. For $f\in H^s$, we compute   
$$     
f^{(0)}(0), \dots , f^{(j-1)}(0)\ \ \  \mbox{and}\ \ \         
f^{(0)}(1), \dots , f^{(j-1)}(1).     
$$     
With this information we define a 
polynomial $p_{f,j}$ of degree at most $j$ such that        
$\tilde f_j=f-p_{f,j}$ is a periodic function from $\wt H^j$.     
To obtain the polynomial $p_{f,j}$,        
we use     
the normalized Bernoulli polynomials from \eqref{eq:bernoulli}.      
In particular, $B^*_0(x)=1$ and $B^*_1(x)=x-\tfrac12$.      
For $m\ge1$, we have       
$[B^*_{m}]'= B^*_{m-1}$ which yields       
\begin{equation}\label{ber0}       
[B^*_m]^{(\ell)}=B^*_{m-\ell}\ \ \ \ \     
\mbox{for all}\ \ \ \ \ \ell=0,1,\dots,m.       
\end{equation}       
Furthermore,       
\begin{equation}\label{ber2}       
B^*_1(1)-B^*_1(0)=1\ \ \ \ \  \mbox{and}\ \ \ \ \     
B^*_m(1)-B^*_m(0)=0\ \ \ \mbox{for all}\ \ \ m\neq1.       
\end{equation}       
       
For $f\in H^s\subset H^j$, we define the polynomials $p_{f,j}$ by        
\begin{equation}    \label{eq:poly}       
p_{f,j}(x) \;:=\; \sum_{m=0}^{j-1}\,        
\bigl(f^{(m)}(1)-f^{(m)}(0)\bigr)\,B^*_{m+1}(x).       
\end{equation}       
We stress that the computation of the value $p_{f,j}(x)$      
requires the $2j$ values       
of $f^{(m)}(1)$ and $f^{(m)}(0)$ for~$m=0,1,\dots,j-1$.         
       
For $\ell=0,1,\dots,j-1$, we conclude from~\eqref{ber0} that        
$$       
p_{f,j}^{(\ell)}(x)=\sum_{m=\max(0,\ell-1)}^{j-1}       
\bigl(f^{(m)}(1)-f^{(m)}(0)\bigr)\,B^*_{m+1-\ell}(x).       
$$       
Using~\eqref{ber2} we obtain        
\[       
p_{f,j}^{(\ell)}(1) - p_{f,j}^{(\ell)}(0) \;=\; f^{(\ell)}(1) - f^{(\ell)}(0)       
\ \ \ \ \        
\mbox{for all}\ \ \ \ \ \ell=0,1,\dots,j-1.        
\]       
This implies that $f-p_{f,j}\in \wt H^j$ for all $f\in H^s$.       
     
Since $f-p_{f,j}\in\wt H^j$ and the norm of $I_k$ restricted to      
the space $\wt H^j$ is given by     
Proposition~\ref{prop1} with $s$ replaced by $j$, we know that     
$$     
|I_k(f)-I_k(p_{f,j})|\;\le\;      
\left\|I_k\big|_{\wt H^j}\right\|\,\|f-p_{f,j}\|_{H^j}     
\;=\; \frac{\|f-p_{f,j}\|_{H^j}}{\max\{1,(2\pi|k|)^j\}}     
$$     
and, by Parseval's identity (in $\wt H^j$), that     
\begin{equation}\label{eq:parseval}     
\begin{split}      
\|f-p_{f,j}\|_{H^j}^2      
\;&=\; \abs{\il f-p_{f,j}, 1\ir_j}^2    
\;+\; \sum_{\ell\in\Z\setminus0} \abs{\il f-p_{f,j}, e_\ell\ir_j}^2 
(2\pi \ell)^{-2j}\\      
\;&=\; \abs{\il f, 1\ir_j}^2      
\;+\; \sum_{\ell\in\Z\setminus0} \abs{\il f, e_\ell\ir_j}^2 (2\pi \ell)^{-2j}
\;\le\; \|f\|_{H^j}^2 \;\le\; \|f\|_{H^s}^2.      
\end{split}     
\end{equation}     
Here we used the fact that $\il p_{f,j}, e_k\ir_j=0$ for all $k\in\Z$.     
This proves that     
$$     
|I_k(f)-I_k(p_{f,j})|\;\le\;      
\frac{\|f\|_{H^s}}{\max\{1,(2\pi|k|)^j\}}.    
$$     
Note that the last      
upper bound is not small for $k=0$. However, if $k\not=0$ then     
for $j\in[1,s]$ for all $f\in H^s$ we have     
\begin{equation}\label{knotzero}     
|I_k(f)-I_k(p_{f,j})|\;\le\; \frac{\|f\|_{H^s}}{(2\pi|k|)^j},     
\end{equation}      
which is exponentially small in $j$.      
     
We now show how to compute $I_k(p_{f,j})$ exactly. Indeed,        
$$       
I_k(p_{f,j})=\sum_{m=0}^{j-1}\bigl(f^{(m)}(1)-f^{(m)}(0)\bigr)\,       
I_k(B^*_{m+1}),       
$$       
and it is enough to compute $I_k(B^*_{m+1})$.       
For $k=0$ we have $I_0(B^*_{m+1})=\int_0^1B^*_{m+1}(x)\,{\rm d}x=0$        
for all $m=0,1,\dots, j-1$. Hence, $I_0(p_{f,j})=0$.       
       
For $k\not=0$, we use the Fourier       
expansion of the normalized Bernoulli polynomials $B^*_{m+1}$,       
\begin{equation}\label{berfou}     
B^*_{m+1}(x)=-\frac{1}{(2\pi\,i\,)^{m+1}}\,       
\sum_{\ell\in\Z\setminus 0}\frac{e^{2\pi\,i\,\ell\,x}}{\ell^{j+1}}       
\ \ \ \ \ \mbox{for all}\ \ \ \ \ x\in[0,1].       
\end{equation}     
This yields       
$$       
I_k(B^*_{m+1})=-\frac{1}{(2\pi\,i)^{m+1}}\,       
\sum_{\ell\in\Z\setminus 0}\frac1{\ell^{j+1}}\,       
\int_0^1e^{2\pi\,i\,(\ell-k)\,x}\,{\rm d}x=       
-\frac{1}{(2\pi\,i k)^{m+1}}.       
$$       
Hence,       
$$       
I_k(p_{f,j})=       
\begin{cases}       
\ 0&\ \ \mbox{for}\ \ k=0\\       
\       
-\sum_{\ell=0}^{j-1}\frac{f^{(\ell)}(1)-f^{(\ell)}(0)}{(2\pi\,i k)^{\ell+1}}&       
\ \ \mbox{for}\ \ k\not=0.       
\end{cases}       
$$       
For $k\neq0$,     
the computation of $I_k(p_{f,j})$ requires       
the $2j$ values of $f^{(\ell)}(1)$ and      
$f^{(\ell)}(0)$ for $\ell=0,1,\dots,j-1$       
which are also needed for the computation of $p_{f,j}(x)$.

\subsection{Taylor's  Expansion}\label{subsec:taylor}     
For $n\in[1, s]$, we use Taylor's expansion of $f\in H^s$ at $\tfrac12$. Let     
$$     
T_{f,n}(x)=f(\tfrac12)+f^{\prime}(\tfrac12)(x-\tfrac12)+\cdots+\frac{f^{(n-1)}(\tfrac12)}{(n-1)!}\,(x-\tfrac12)^{n-1}     
\ \ \ \ \ \mbox{for all}\ \ \ \ \ x\in[0,1].     
$$       
Then     
$$     
f(x)-T_{f,n}(x)=     
\frac{(x-\tfrac12)^n}{(n-1)!}\,\int_0^1(1-t)^{n-1}\,f^{(n)}\left(\tfrac12+t(x-\tfrac12)\right)     
\,{\rm d}t     
\ \ \ \ \ \mbox{for all}\ \ \ \ \ x\in[0,1].     
$$      
This allows us to estimate $I_k(f-T_{f,n})$ since     
$$     
I_k(f-T_{f,n})= \frac1{(n-1)!}\,\int_0^1\int_0^1     
\,\eu^{-2\pi\,i\,kx} \,      
(x-\tfrac12)^n\,(1-t)^{n-1}\,f^{(n)}\left(\tfrac12+t(x-\tfrac12)\right)     
\,{\rm d}t\, {\rm d} x,     
$$     
and      
$$     
|I_k(f-T_{f,n})|\le \frac1{(n-1)!}\,\int_0^1\int_0^1     
|x-\tfrac12|^n\,|f^{(n)}\left(\tfrac12+t(x-\tfrac12)\right)|     
\,{\rm d}t\, {\rm d} x.     
$$     
We now change variables $y=\tfrac12+t(x-\tfrac12)\in[0,1]$  
     so that ${\rm d}y=     
(x-\tfrac12){\rm d}t$ and then     
$$     
|I_k(f)-I_k(T_{f,n})|\le      
\frac1{(n-1)!}\,\int_0^1\int_0^1     
|x-\tfrac12|^{n-1}\,|f^{(n)}(y)|     
\,{\rm d}x\, {\rm d} y \le \frac1{2^{n-1}\,n!}\,\|f^{(n)}\|_{L_2}.     
$$     
This proves that for $n\in[1, s]$ and all $f\in H^s$ we have      
\begin{equation}\label{Taylor}     
|I_k(f)-I_k(T_{f,n})|\le      
\frac1{2^{n-1}\,n!}\,\|f\|_{H^s}.      
\end{equation}     
     
Furthermore, we can compute $I_k(T_{f,n})$ exactly if we know      
$f(\tfrac12),f^{\prime}(\tfrac12),\dots,f^{(n-1)}(\tfrac12)$. Indeed,     
$$     
I_k(T_{f,n})=\sum_{\ell=0}^{n-1}f^{(\ell)}(\tfrac12)\,     
\frac1{\ell!}\,\int_0^1\,\eu^{-2\pi\,i\,kx} \, (x-\tfrac12)^{\ell}\,{\rm d}x.     
$$     
For $k=0$, we have      
\begin{equation}\label{taylorko}     
I_0(T_{f,n})=\sum_{\ell=0}^{n-1}f^{(\ell)}(\tfrac12)\,     
\frac{1+(-1)^\ell}{2^{\ell+1}(\ell+1)!}.     
\end{equation}     
{}For $k\not=0$, we use integration by parts and show that     
$$     
\frac1{\ell!}\,\int_0^1\,\eu^{-2\pi\,i\,kx} \, (x-\tfrac12)^{\ell}\,{\rm d}x=     
\frac1{(2\pi\,i\,k)^{\ell+1}}\,\sum_{m=0}^{\ell}\frac{i^m\,(k\pi)^m((-1)^m-1)}{m!}.     
$$     
Hence for $k\not=0$, we have     
\begin{equation}\label{taylorknoto}     
I_k(T_{f,n})=\sum_{\ell=0}^{n-1}\frac{f^{(\ell)}(\tfrac12)}     
{(2\pi\,i\,k)^{\ell+1}}\,\sum_{m=0}^{\ell}\frac{i^m\,(k\pi)^m((-1)^m-1)}{m!}.     
\end{equation}     
     
\subsection{Algorithms}\label{subsec:algorithms}     
     
With the preparations done in the previous two subsections,      
we are ready to define      
algorithms for the non-periodic case.      
     
\begin{itemize}     
\item     
Assume first that $k\in\Z\setminus\{0\}$.     
\end{itemize}      
      
We discuss algorithms based on periodization for $f\in H^s$.     
We define the algorithm $A_n^{\rm Per}$      
for all even $n\in[2,2s)$     
and for $n=2s+\ell$ with $\ell \in \N_0$.     
     
For even $n\in[2,2s)$ we compute     
$f^{(0)}(0), \dots,f^{((n-2)/2)}(0)$,      
$f^{(0)}(1), \dots,f^{((n-2)/2)}(1)$, and define      
$$     
A_{n}^{\rm Per}(f)=I_k(p_{f,n/2})     
$$     
with $p_{f,n/2}$ given by~\eqref{eq:poly} for $j=n/2\le s$.      
     
For $n= 2s+\ell$ with $\ell \in \N_0$, we compute      
$f^{(0)}(0), \dots,f^{(s-1)}(0)$,      
$f^{(0)}(1), \dots,f^{(s-1)}(1)$ to obtain the polynomial $p_{f,s}$.     
Then we define      
$$     
A_n^{\rm Per}(f)=I_k(p_{f,s})+A^*_{\ell+1}(f-p_{f,s}) 
$$     
with the algorithm $A^*_{\ell+1}$ from Section~\ref{sec:periodic}     
defined by~\eqref{eq:A_per}.      
The algorithm $A^*_{\ell+1}$ uses no extra 
information on $f$ if $\ell<2|k|-1$.     
For $\ell\ge2|k|-1$, the algorithm $A^*_{\ell+1}$      
uses extra $\ell$ function values at      
$j/(\ell+1)$ for $j=1,2,\dots,\ell$. Note that we have already computed     
the function value at $j/(\ell+1)$ for $j=\ell+1$.      
      
We stress that the algorithm $A_{n}^{\rm Per}$ is well defined 
for $n=2s+\ell$       
since $f-p_f\in\wt H^s$ and $\wt H^s$ is the domain of the algorithm       
$A^*_{\ell+1}$. For $f\in\wt H^s$ we have $p_{f,j}=0$ for all $j\in[1,s]$,      
and therefore $A_n^{\rm Per}(f)=0$ for all even $n\in[2,2s)$ and     
$A_n^{\rm Per}(f)=A^*_{\ell+1}(f)$ for $n=2s+\ell$.      
       
The algorithm $A_n^{\rm Per}$ uses at most $n$ evaluations of $f$.     
Indeed, for even $n\in[2,2s)$ it uses $n/2$ evaluations at the endpoint     
points $x=0$ and $x=1$, so that the total number is~$n$.      
For $n=2s+\ell$, the algorithm $A_n^{\rm Per}$ uses two function      
values and $2(s-1)$ values of derivatives     
of $f$ at $x=0$ and $x=1$, as well as at most $\ell$      
functions values at $j/(\ell+1)$ for $j=1,\dots,\ell$,      
which is $2+2(s-1)+\ell=2s+\ell=n$, as claimed.      
     
{}From  the formulas of Sections~\ref{sec:periodic}      
and~\ref{subsec:periodization} we find the explicit form of $A_n^{\rm Per}$.       
{}For even $n\in[2,2s)$ and $n=2s+\ell$ with $\ell<2|k|-1$ we have     
$$     
A_n^{\rm Per}(f)=\sum_{j=0}^{(n-2)/2}\frac{f^{(j)}(0)-f^{(j)}(1)}{(2\pi i\,k)^{j +1}},      
$$     
whereas for $n=2s+\ell$ with $\ell\ge2|k|-1$, we have      
\begin{eqnarray*}     
A_n^{\rm Per}(f)&=&     
\sum_{j=0}^{s-1}\frac{f^{(j)}(0)-f^{(j)}(1)}{(2\pi i\,k)^{j +1}}\,+\\     
\,&&     
\frac1{\ell+1}\,\sum_{j=1}^{\ell+1}\left(f\left(\frac{j}{\ell+1}\right)-p_{f,s}\left(\frac{j}{\ell+1}\right)\right)\,       
\exp^{-2\pi\,i\,k\,j/(\ell+1)}.       
\end{eqnarray*}     
     
Note that for $s=1$ the algorithm $A_{n}^{\rm Per}$ uses only       
function values since $p_{f,1}(x)=(f(1)-f(0))(x-\tfrac12)$,      
whereas for $s\ge2$ it also uses derivatives of $f$.       
The weights used by the algorithm $A_{n}^{\rm Per}$ are complex.       
However, the sum of their absolute values is bounded by an absolute       
constant independent of $n$ since it is 
known that the values of the normalized       
Bernoulli polynomials $B^*_j$, which are 
present in $p_{f,s}$, are exponentially       
small in~$j$. This implies numerical stability of the algorithm~$A_{n}$.        
       
Obviously, the derivatives $f^{(j)}(0)$ and $f^{(j)}(1)$ for       
$j=1,2,\dots,s-1$ can be approximated 
by computing $2(s-1)$ extra function values;
the worst case error of this approximation
for functions from the unit ball of $H^s$ can be 
made arbitrarily small. 
Hence, there are algorithms       
that use only $n$ function values and they have a worst case error         
arbitrarily close to the worst case error of $A_{n}^{\rm Per}$.         
However, stability of such algorithms is not clear since then we must       
use huge coefficients.  
We leave it as an open problem for $s\ge2$ if there are       
stable algorithms that use only $n$ function values and whose worst case       
error are comparable to the algorithm $A_{n}^{\rm Per}$.          
       
We are ready to bound the worst case error of $A_{n}^{\rm Per}$.       
     
\begin{thm}\label{thm:ans}       
For $k\not=0$, we have      
\begin{itemize}     
\item for even $n\in[2,2s]$,     
$$     
e(A_n^{\rm Per})\le     
\frac{1}{(2\pi |k|)^{n/2}},     
$$     
\item for $n>2s$,      
$$     
e(A_n^{\rm Per})\le     
\left(\frac3{2\pi}\right)^s\,\frac{2}{(n-2s+1+|k|)^s}.     
$$     
\end{itemize}       
\qed     
\end{thm}       
     
\begin{proof}       
{}For even $n\in[2,2s]$, we have     
$$     
I_k(f)-A_n^{\rm Per}(f)=I_k(f)-I_k(p_{f,n/2})     
$$     
and~\eqref{knotzero} implies the bound on $e(A_n)$.      
     
{}For $n=2s+\ell$, we clearly have      
$f=f-p_{f,s}+p_{f,s}$ for all $f\in H^s$.      
By definition of $A_n^{\rm Per}$ and linearity of $I_k$ we obtain     
\[     
I_k(f)-A_n^{\rm Per}(f)     
\,=\, I_k(f-p_{f,s})-A^*_{\ell+1}(f-p_{f,s}).  
\]     
{}From~\eqref{estanstar} we know that     
$$     
|I_k(f-p_{f,s})-A^*_{\ell+1}(f-p_{f,s})|\le      
\left(\frac3{2\pi}\right)^s\,\frac2{(n-2s+1+|k|)^s}\,\|f-p_{f,s}\|_{H^s}.     
$$     
Then~\eqref{eq:parseval} with $j=s$ 
yields $\|f-p_{f,s}\|_{H^s}\le\|f\|_{H^s}$,     
which implies the bound on $e_n(A_n^{\rm Per})$.     
This completes the proof.     
\end{proof}       
     
\begin{itemize}      
\item     
Assume now that $k\in\Z$.     
\end{itemize}      
     
Although $k$ is now an arbitrary integer, our emphasis      
will be later on $k=0$ or, more generally, on $|k|$ small relative to $n$.      
We discuss algorithms based on Taylor's expansion and 
periodization for $f\in H^s$.     
We define the algorithm $A_n^{\rm Tay-Per}$      
for all  $n\in[1,s]$ and for $n=2s+\ell$ with $\ell \in \N_0$.      
     
For $n\in[1,s]$, we compute $f^{(0)}(\tfrac12),\dots, f^{(n-1)}(\tfrac12)$     
and define     
$$     
A_n^{\rm Tay-Per}(f)=I_k(T_{f,n}),     
$$     
where $T_{f,n}$ is Taylor's expansion of $f$ at $\tfrac12$ up 
to the $(n-1)$st derivative and     
$I_k(T_{f,n})$ is explicitly given by~\eqref{taylorko} for      
$k=0$ and by~\eqref{taylorknoto} for $k\not=0$.     
For $n=2s+\ell$ with $\ell \in \N_0$, we define     
$$     
A_n^{\rm Tay-Per}(f)=A_n^{\rm Per}(f),     
$$     
where $A_n^{\rm Per}$ is from the previous subsection.      
     
Clearly, the algorithm $A_n^{\rm Tay-Per}$ uses 
at most $n$ evaluations of $f$.      
For $s=1$ it uses only function values and it is defined for all $n$.     
For $s\ge2$, it also uses derivatives of $f$ and it is not defined for     
$n\in[s+1,2s-1]$. 
Its weights are complex and the sum of their absolute values      
is uniformly bounded in $n$. Hence, $A_n^{\rm Tay-Per}$ is 
stable, and a similar     
remark on the approximation of derivatives by function values can be made      
as for the algorithm $A_n^{\rm Per}$.      
     
The worst case error of $A_n^{\rm Tay-Per}$ can be easily bounded       
by~\eqref{Taylor} and Theorem~\ref{thm:ans}. We summarize these bounds      
in the following theorem.     
     
\begin{thm}\label{thm:ans2}      
For an arbitrary integer $k$, we have     
\begin{itemize}     
\item     
for $n\in[1,s]$,     
$$     
e(A_n^{\rm Tay-Per})\le \frac1{2^{n-1}\,n!},     
$$     
\item     
for $n\ge 2s$      
$$     
e(A_n^{\rm Tay-Per})\le       
\left(\frac3{2\pi}\right)^s\,\frac{2}{(n-2s+1+|k|)^s}.     
$$     
\end{itemize}     
\end{thm}     
     
We now comment on Theorems~\ref{thm:ans} and~\ref{thm:ans2} for a finite $s$.     
For $k=0$ and initial $n$, i.e., even $n\le2s$ or $n\le s$,      
we can only apply Theorem~\ref{thm:ans2}. It tells us that for  $n\in[1,s]$     
the error bound of $A_n^{\rm Tay-Per}$ is exponentially small in $n$.      
Note that for non-zero $k$     
we can use both theorems. For the initial $n$ and $|k|$ small relative      
to $n$, the first bound     
of Theorem~\ref{thm:ans2} is smaller than the first bound in Theorem~\ref{thm:ans}.     
On the other hand, for large $|k|$ relative to~$n$, the opposite is true.      
Obviously for $n>2s$, both theorems coincide and the error bound of $A_n^{\rm Tay-Per}=A_n^{\rm Per}$     
is of the form     
\begin{equation}\label{csexp}     
e(A_n^{\rm Per})\,\le\, \left(\frac3{2\pi}\right)^s \frac{2}{(n-2s+1+|k|)^s}.     
\end{equation}     
     
The last bound yields an upper bound on the $n$th minimal error      
$e(n,k,s)$ for $n\ge2s$.      
Combining this with~\eqref{lownonper} and Theorem~\ref{thm:int1}     
we obtain sharp lower and upper bounds on the minimal errors     
$e(n,k,s)$.      
\begin{thm} \label{thm:int2}        
Consider the integration problem $I_k$ defined over the        
space $H^s$ of non-periodic functions with $s \in \nat$.        
Then       
$$        
\frac{c_s}{(n+|k|)^s} \;\le\; e(n,k,s) \;\le\;        
\left(\frac3{2\pi}\right)^s \frac{2}{(n+|k|-2s+1)^s} ,        
$$        
for all $k\in\Z$ and $n\ge 2s$. The positive number $c_s$ is     
from Theorem~\ref{thm:int1}.       
\qed      
\end{thm}         
     
We stress that Theorem~\ref{thm:int2} presents      
an upper bound on the minimal       
errors $e(n,k,s)$ only for $n\ge2s$,     
the lower bound holds for all $n$.      
The reason is that we need $2s$       
function and derivatives values to periodize the function $f$ which       
enables us to use the algorithm $A_n^*$. We do not know sharp bounds       
on $e(n,k,s)$ for $n\in[1,2s)$. However, we know that      
$e(n,k,s)$ is at most $1/(2^{n-1}n!)$ for $n\le s$ and all $k$, see Theorem~\ref{thm:ans2},     
and at most  $(2\pi|k|)^{-n/2}$ for $n\le 2s$ and all $k\not=0$,      
see Theorem~\ref{thm:ans}.     
Of course, the problem of the minimal errors $e(n,k,s)$ for initial $n$ it is not very important as       
long as $s$ is not too large.       
       
The minimal errors $e(n,k,s)$ for the non-periodic case have a       
peculiar property for $s\ge2$ and large $k$. Namely,        
for $n=0$ we obtain the initial error which is of order $|k|^{-1}$,       
whereas for $n\ge 2s$ it becomes of order $|k|^{-s}$. Hence, the       
dependence on $|k|^{-1}$ is short-lived and disappears quite quickly.       
For instance, take $s=2$. Then $e(n,k,s)$ is of order $|k|^{-1}$        
only for $n=0$ and maybe for $n=1,2,3$, and then becomes of order $|k|^{-2}$.        
     
\vskip 1pc       
We now briefly discuss the absolute and normalized error       
criteria for the non-periodic case.        
For the \emph{absolute error criterion}, the information        
complexity~$n^{\rm abs}(\e,k,s)$ for $\e\in(0,1)$ is defined as        
$$        
n^{\rm abs}(\e,k,s)        
=\min\left\{\,n\ |\ \ e(n,k,s)\le \e\,\right\}.        
$$        
Clearly, $n^{\rm abs}(\e,k,s)=0$         
for $\e\ge  e(0,k,s)$.  For $\e< e(0,k,s)$        
we can bound $n^{\rm abs}(\e,k,s)$  by Theorem~\ref{thm:int2}.       
This implies the following corollary.        
       
\begin{cor}        
Consider the absolute error criterion for the integration problem       
$I_k$ defined over the space $H^s$.        
Let $c_s$ be from Theorem~\ref{thm:int1}.       
\begin{itemize}        
\item For $k=0$ and all $\e\in(0,1)$ we have         
$$        
c_s^{1/s}\,\left(\frac1{\e}\right)^{1/s}        
\,\le\, n^{\rm abs}(\e,0,s) \,\le\,        
\left\lceil     
\left(\frac3{2\pi}\right)\left(\frac{2}{\e}\right)^{1/s}\right\rceil +2s-1.       
$$        
     
\item For $k\not=0$ and $\e\in[\beta_{k,s}/(2\pi|k|),1)$,       
with $\beta_{k,s}$ from Proposition~\ref{prop:non_initial},         
we have        
$$        
n^{\rm abs}(\e,k,s)=0,        
$$        
whereas for $\e\in(0,\beta_{k,s}/(2\pi|k|))$ we have        
$$        
c_s^{1/s}\,\left(\frac1{\e}\right)^{1/s}\,-\,|k|        
\,\le\, n^{\rm abs}(\e,k,s) \,\le\,      
2s \,+\, \max\left\{0,      
\left\lceil \left(\frac3{2\pi}\right)\left(\frac{2}{\e}\right)^{1/s}\right\rceil     
 -1-|k| \right\}.        
$$        
\end{itemize}        
\qed       
\end{cor}        
     
Similarly as for the periodic case, this means that        
for the absolute error criterion the problem for the non-periodic case       
becomes easier for large $|k|$. However, for $k\not=0$,        
the condition on $\e$ is now quite different for $s\ge2$       
as compared to the periodic case, see Corollary~\ref{coro:per_abs_n}.       
We also stress that the asymptotic behaviors of        
$\tilde n^{\rm abs}(\e,k,s)$ and $n^{\rm abs}(\e,k,s)$        
are of order~$\e^{-1/s}$ and do not depend on $k$.       
       
We now turn to the \emph{normalized error criterion} for which        
the information complexity  $n^{\rm nor}(\e,k,s)$ for $\e\in(0,1)$        
is defined as         
$$        
\tilde n^{\rm nor}(\e,k,s)        
=\min\left\{\,n\ |\ \ e(n,k,s)\le \e\,e(0,k,s)\,\right\}.        
$$        
We always have $n^{\rm nor}(\e,k,s)\ge1$. For $k=0$        
we have $e(0,0,s)=1$ and there is no difference        
between the normalized and absolute error criteria also for       
the non-periodic case.        
        
For $k\not=0$ the situation is quite different. From        
Theorem~\ref{thm:int2}, Proposition~\ref{prop1}     
as well as the estimates of $\beta_{k,s}$,         
it is easy to obtain the following corollary.        
       
\begin{cor}\label{cor157}        
Consider the normalized error criterion for the integration problem       
$I_k$ defined over the space $H^s$.        
Let $c_s$ be from Theorem~\ref{thm:int1}.      
       
For all $k\not=0$ and all $\e\in(0,1)$ we have        
$$        
c_s^{1/s}\,\left(\frac{\sqrt{2}\pi|k|}{\e}\right)^{1/s}-|k|       
\le n^{\rm nor}(\e,k,s) \le \;       
 2s + \max\left\{0,      
        \left\lceil \left(\frac3{2\pi}\right)\left(\frac{4\pi|k|}{\e}\right)^{1/s}\right\rceil     
        -1-|k|\right\},       
$$       
which can be written as     
\begin{equation}\label{eq:compl_nor_non-per}      
n^{\rm nor}(\e,k,s)=\Theta\left(\frac{|k|^{1/s}}{\e^{1/s}}\right)       
\ \ \ \ \ \mbox{as}\ \ \ \ \ \e\to0.       
\end{equation}      
\end{cor}        
The asymptotic expression \eqref{eq:compl_nor_non-per} shows       
that for the normalized error criterion the problem becomes harder for      
large $|k|$ and small $\eps$.       
The dependence on $k$ is through $|k|^{1/s}$        
and decreases with $s$. This should be compared with the periodic case,       
where the dependence on $|k|$ is linear.       
Hence, the dependence on $k$ is the same for $s=1$,       
and the periodic case is harder than the non-periodic       
case for $s\ge2$ and small $\eps$.      
      
For fixed $\eps$ and varying $|k|$, the difference in the behavior of      
the information complexity in $|k|$ is even more dramatic and depends on $s$.      
Consider first $s=1$.      
Then Corollary~\ref{cor157} yields for $\e<\sqrt{2}\pi c_s$ that     
$$      
\lim_{|k|\to\infty}n^{\rm nor}(\e,k,s)=\infty,      
$$      
as for the periodic case, see~\eqref{pernor}.      
      
Assume now that $s\ge2$.       
In this case, the information complexity for the non-periodic problem       
does not go to infinity with $|k|$ in contrast to the periodic case,      
see again~\eqref{pernor}. This simply follows from       
Corollary~\ref{cor157} since the second term of the maximum behaves      
like $\mathcal{O}(|k|^{1/s})-|k|$ and goes to $-\infty$. Hence      
\begin{equation}\label{cor158}      
\limsup_{|k|\to\infty}\, n^{\rm nor}(\e,k,s)\le 2s.      
\end{equation}      
This is even true if we choose $\eps$ slowly decreasing with $|k|$,       
say $\eps_k=|k|^{-\eta}$ for some~$\eta\in(0,s-1)$. Indeed, then       
$|k|/\e_k=|k|^{1+\eta}$ and $\mathcal{O}(|k|^{(1+\eta)/s})-|k|$      
still goes to $-\infty$ and~\eqref{cor158} is again valid.        
This discussion can be summarized as follows.      
\begin{cor}      
For the non-periodic case and the normalized error criterion        
\begin{itemize}      
\item      
for $s\ge1$,       
oscillatory integration becomes       
harder in $|k|$ asymptotically in $\e$,      
\item      
for $s=1$ and fixed small $\e$,       
oscillatory integration becomes       
harder in $|k|$,      
\item       
for $s\ge2$ and fixed $\e$ or      
even for $\e^{-1}=\mathcal{O}(|k|^\eta)$ with $\eta\in(0,s-1)$,       
oscillatory integration becomes easy since      
$n^{\rm nor}(\e,k,s)$ is at most $2s$ for large $|k|$.      
\end{itemize}      
\end{cor}      
        
\section{The case of $s=\infty$}       
We briefly discuss the space $H^\infty$ which is defined as      
$$     
H^\infty=\{f\in C^\infty([0,1])\,| \ \      
\sum_{\ell=0}^\infty \|f^{(\ell)}\|_{L_2}^2<\infty\,\},     
$$     
where $\|f^{(\ell)}\|_{L_2}$  
denotes the     
$L_2=L_2([0,1])$ norm of      
$f^{(\ell)}$.  Note that $H^\infty$ consists of infinitely      
many times differentiable functions.     
In particular, all polynomials belong to $H^\infty$     
but $e_h(x)=\exp(2\pi ihx)$ belongs to $H^\infty$ iff $h=0$.      
     
We equip the space $H^\infty$ with the two inner products     
\begin{eqnarray*}     
\il f,g\ir_{\infty}&=&\sum_{\ell=0}^{\infty} \il f^{(\ell)},1\ir_0      
\widebar{\il g^{(\ell)},1\ir_0},\\     
\il f,g\ir_{\infty,*}&=&\sum_{\ell=0}^{\infty} \il     
f^{(\ell)},g^{(\ell)}\ir_0.     
\end{eqnarray*}     
As for a finite $s$, the norms generated by theses inner products are     
closely related since we have     
$$     
\tfrac{12}{13}\,\|f\|_{H^\infty_*}\le \|f\|_{H^\infty}\le \|f\|_{H^\infty_*}    
\ \ \ \ \mbox{for all}\ \ \ f\in H^\infty,     
$$     
see the appendix. This means that it is enough to consider only one of     
these inner product and, as before, we choose $\il     
\cdot,\cdot\ir_\infty$  for simplicity of the analysis.     
    
\begin{prop}\label{dense1}     
Polynomials are dense in $H^\infty$, i.e.,      
for any $f\in H^\infty$ and any positive $\e$      
there is a polynomial $p$ such that     
$$     
\|f-p\|_{H^\infty}\le \e.     
$$     
\end{prop}     
\begin{proof}     
We begin by showing that for an absolutely      
continuous function $g$ for which $g^{\,\prime}\in L_2([0,1])$     
we have       
\begin{equation}\label{dense2}     
\|g-g(\tfrac12)\|_{L_2}\le \tfrac12\,\|g^{\,\prime}\|_{L_2}.     
\end{equation}     
Indeed, $g(x)-g(\tfrac12)=\int_{1/2}^xg^{\,\prime}(t)\,{\rm d}t$ and     
$$     
|g(x)-g(\tfrac12)|\le \int_{1/2}^x|g^{\,\prime}(t)|     
\,{\rm d}t 
\le \left|\int_{1/2}^x{\rm d}t\right|^{1/2}\, 
\left|\int_{1/2}^x|g^{\,\prime}(t)|^2\,{\rm d}t\right|^{1/2}     
\le|x-\tfrac12|^{1/2}\,\|g^{\,\prime}\|_{L_2}.      
$$     
Hence,     
$$     
\|g-g(\tfrac12)\|_{L_2}     
\le \left(\int_0^1|x-\tfrac12|\,{\rm d}x\right)^{1/2}\,\|g^{\,\prime}\|_{L_2}
= \tfrac12\,\|g^{\,\prime}\|_{L_2},     
$$     
as claimed.      
     
Take now an arbitrary $f\in H^\infty$. For any positive $\delta$      
there exists $\ell^*=\ell^*(f,\delta)\in \natural$ such that     
$$     
\sum_{\ell=\ell^*}^\infty\|f^{(\ell)}\|^2_{L_2}\le \delta^2.     
$$     
In particular, $\|f^{(\ell^*)}\|_{L_2}\le \delta$.      
Taking $g^{\,\prime}=f^{(\ell^*)}$     
we conclude from~\eqref{dense2} that       
$$     
\|f^{(\ell^*-1)}-f^{(\ell^*-1)}(\tfrac12)\|_{L_2}\le \tfrac12\,\delta.     
$$     
For $\ell^*\ge2$, 
we take $g^{\,\prime}=f^{(\ell^*-1)}-f^{(\ell^*-1)}(\tfrac12)$ 
and we have again     
from~\eqref{dense2}     
$$     
\|f^{(\ell^*-2)}-f^{(\ell^*-1)}(\tfrac12)(\cdot-\tfrac12)-f^{(\ell^*-2)}(\tfrac12) \|_{L_2}     
\le \tfrac14\,\delta.     
$$     
Repeating this procedure we conclude that for     
$$     
p(x)=f(\tfrac12)+f^{(\prime)}(\tfrac12)(x-\tfrac12)+\cdots      
+\frac{f^{(\ell^*-1)}(\tfrac12)}{(\ell^*-1)!}     
(x-\tfrac12)^{\ell^*-1}     
$$     
we have      
$$     
\|f^{(\ell)}-p^{(\ell)}\|_{L_2}\le 2^{\ell^-\ell^*}\,\delta\ \ \ \ \ \mbox{for all}\ \ \ \ \     
j=0,1,\dots,\ell^*-1.     
$$     
Hence,     
\begin{eqnarray*}     
\|f-p\|^2_{H^\infty}&\le&\sum_{\ell=0}^\infty\|f^{(\ell)}-p^{(\ell)}\|^2_{L_2}     
=\sum_{\ell=0}^{\ell^*-1}\|f^{(\ell)}-p^{(\ell)}\|^2_{L_2}\ +\      
\sum_{\ell=\ell^*}^\infty\|f^{(\ell)}-p^{(\ell)}\|^2_{L_2}\\     
&\le& \delta^2\,\sum_{\ell=0}^{\ell^*-1}4^{\ell-\ell^*}\ +\ \delta^2=\tfrac43\,\delta^2.     
\end{eqnarray*}     
Taking $\delta=\sqrt{3/4}\,\e$, Proposition~\ref{dense1} is proved.     
\end{proof}     
     
It is easy to see that that the periodic subspace      
$$     
\wt H^\infty=\{\,f\in H^\infty\,|\ f^{(\ell)}(0)=f^{(\ell)}(1)\ \      
\mbox{for}\  \ \ell \in \N_0  \}     
$$     
consists only of constant functions. Indeed, since      
$\|I_k\|_{\wt H^\infty\to\C}\le \|I_k\|_{\wt H^s\to\C}$      
for all $s\in\natural$ and for $k\not=0$ we have $\|I_k\|_{\wt H^s\to\C}=(2\pi|k|)^{-s}$,      
we conclude that $I_k=0$ for all $k\not=0$.     
This means that $f\in\wt H^\infty$ implies that $f=\mbox{constant}$, as claimed.      
It is also easy to check that the reproducing kernel of $\wt H^\infty$ is      
$\wt K_\infty(x,t)=1$.      
     
Let $\tilde e(n,k,\infty)$ be the minimal errors  for $\wt H^\infty$. Then     
$\tilde e(0,0,\infty)=1$ and $\tilde e(n,0,\infty)=0$ for all $n\ge1$, whereas      
$\tilde e(n,k,\infty)=0$ for all $n\ge0$ and $k\not=0$.     
     
This means that the periodic case is trivial and cannot be used     
as a tool for the non-periodic case. That is why our lower bound      
on $\tilde e(n,k,s)$ which was quite useful      
for a finite~$s$ is meaningless for $s=\infty$.     
In fact, the problem of non-trivial lower bounds for $H^\infty$ is open.       
     
Proposition~\ref{dense1} together with \eqref{eq:bernoulli2} shows that      
the set of normalized Bernoulli polynomials $\{B^*_j\}_{j=0,1,\dots}$ is a      
complete orthonormal basis of $H^\infty$ and therefore      
the reproducing kernel $K_\infty$ is given by     
$$     
K_\infty(x,t)=\sum_{j=0}^\infty B^*_j(x)\,B^*_j(t)\ \ \ \      
\mbox{for all}\ \ \ \ \ x,t\in [0,1].     
$$     
     
We now present some upper error bounds on     
the minimal errors  $e(n,k,\infty)$ for $H^\infty$.      
In fact, we derived      
the upper bounds in Theorems~\ref{thm:ans}     
and~\ref{thm:ans2} in such a way that they can be used even for $s=\infty$.
     
We start with the initial error. For $k=0$, the representer of $I_0$     
is $1$ and      
$$     
e(0,0,\infty)=1,     
$$     
whereas for $k\not=0$, the proof of      
Proposition~\ref{prop:non_initial} can be modified for $s=\infty$ and      
yields that the representer of $I_k$ is     
$$     
h_{k,\infty}(x)=-\sum_{\ell=1}^\infty(-1)^\ell\,(2\pi     
ik)^{-\ell}\,B^*_\ell(x)     
$$     
and      
$$     
e(0,k,\infty)=\frac{\beta_{k,\infty}}{2\pi|k|}      
$$     
with     
$$     
\beta_{k,\infty}=\left(\frac{4\pi^2 k^2}{4\pi^2 k^2-1}\right)^{1/2}\in[1,1.013].     
$$     
     
For $k=0$ and all $n\ge1$, we can apply the first error      
bound in Theorem~\ref{thm:ans2}     
which states that      
$$     
e(n,0,\infty)\le \frac1{2^{n-1}\,n!},     
$$       
which is super exponentially small in $n$.      
     
For $k\not=0$ and all even $n$, we apply the first error      
bounds in Theorems~\ref{thm:ans} and~\ref{thm:ans2}     
which state that      
$$     
e(n,k,\infty)\le \min\left(\frac1{2^{n-1}\,n!},\,     
\frac{1}{(2\pi |k|)^{n/2}}\right).     
$$     
Note that by Stirling's approximation we have      
\[     
\frac1{2^{n-1}\,n!} \;\le\; 2 \left(\frac{\eu}{2n}\right)^n.     
\]     
It is easy to check that the right hand side is smaller than $\eps$ iff     
\[     
n\bigl(\ln(2n)-1\bigr)\ge \ln(2/\eps)     
\]     
which holds, in particular, if $n\ge2\ln(\eps^{-1})/\ln\ln(\eps^{-1})$      
and $\eps<\eu^{-\eu}=0.135\dots$.     
     
These upper error bounds can be used to estimate the information complexities     
$n^{\rm abs}$ and $n^{\rm nor}$ for $\eps<\eu^{-\eu}$.      
For $k=0$, we have     
$$     
n^{\rm abs}(\e,0,\infty)=     
n^{\rm nor}(\e,0,\infty)     
\,\le\, \left\lceil2\,\frac{\ln\,\e^{-1}}{\ln\,\ln\,\e^{-1}}\right\rceil.     
$$     
For $k\not=0$, we have      
$$     
n^{\rm abs}(\e,k,\infty)\le \left\lceil     
2\,\min\left\{ \frac{\ln\,\e^{-1}}{\ln\,\ln\,\e^{-1}},\;     
\frac{\ln\,\e^{-1}}{\ln\,(2\pi|k|)}\right\} \right\rceil,     
$$     
and     
$$     
n^{\rm nor}(\e,k,\infty)\le      
\left\lceil     
2\,\min\left\{      
\frac{\ln\,\e^{-1}+\ln(2\pi|k|)}{\ln\left(\ln\,\e^{-1}+\ln(2\pi|k|)\right)},\;     
\frac{\ln\,\e^{-1}}{\ln\,(2\pi|k|)}+1\right\} \right\rceil.     
$$     
These estimates are valid for all $\eps<\eu^{-\eu}$. Note that asymptotically,      
when $\e$ tends to zero, all information complexity     
are upper bounded by roughly $\ln(\e^{-1})/\ln(\ln(\e^{-1}))$      
independent of~$k$.

\section{Appendix}     
We now prove~\eqref{equivalence} which shows     
the embeddings constants between the space $H^s$ equipped      
with the norm $\|\cdot\|_{H^s}$ given by~\eqref{eq:inner_product}      
and $\|\cdot\|_{H^s_*}$ given by~\eqref{newnorm1}.      
Since $|\il f^{(\ell)},1\ir_0|^2\le \|f^{(\ell)}\|^2_{L_2}$ we clearly     
have     
$\|f\|_{H^s}\le\|f\|_{H^s_*}$  for all  $f\in H^s$.     
     
To obtain the other estimate, we consider     
$f\in H^s$ which can     be written as     
$$     
f=\sum_{j=0}^s\il f,B^*_j\ir_sB^*_j\,+\,     
\sum_{h\in\Z\setminus{0}}     
\il f,e_h\ir_se_h     
$$     
since the normalized Bernoulli polynomials $B^*_j$ and      
$e^*_h=\exp(2\pi ih\cdot)/(2\pi|h|)^s$ are an orthonormal basis of $H^s$     
with respect to $\il\cdot,\cdot\ir_s$.     
Let $\{b_j\}_{j\in\N}$ be some ordering of this orthonormal basis.    
Then, clearly, we have     
\begin{equation} \label{eq:appendix_Ms}    
\begin{split}    
\|f\|_{H^s_*}^2 \,&=\, \norm{\sum_{j\in\N}\il f,b_j\ir_s b_j}_{H^s_*}^2     
\,=\, \sum_{j,m\in\N}\il f,b_j\ir_s     
\widebar{\il f,b_m\ir_s} \il b_j, b_m\ir_{s,*}\\    
&\le \,\sum_{j,m\in\N}\abs{\il f,b_j\ir_s}^2\abs{\il b_j,b_m\ir_{s,*}}\\    
&\le\, \left(\max_{m\in\N}\left\{\sum_{j\in\N}     
\abs{\il b_j, b_m\ir_{s,*}}\right\} \right)    
        \left(\sum_{j\in\N} \abs{\il f, b_j\ir_{s}}^2\right)\\    
&=\, \left(\max_{m\in\N}\left\{\sum_{j\in\N} 
\abs{\il b_j, b_m\ir_{s,*}}\right\} \right)    
        \|f\|_{H^s}^2\\    
&=:\, M_s\,\|f\|_{H^s}^2.    
\end{split}    
\end{equation}    
To bound $M_s$ we estimate $|\il b_j,b_m\ir_{s,*}|$ for     
all possible $b_j,    
b_m\in\{B^*_j,e^*_h\}_{j=0,1,\dots,s,\,h\in\Z\setminus\{0\}}$.    
    
We start with the case where both $b_j, b_m$     
are in $\{e_h^*\}_{h\in\Z\setminus0}$.     
This case is easy since $\{e_h^*\}_{h\in\Z}$ is also an orthogonal     
basis in $\wt H^s$ with the inner product $\il\cdot,\cdot\ir_{s,*}$.    
We have for $h\in\Z\setminus0$ and $\ell\neq h$ that    
$$     
\il e^*_h, e^*_\ell\ir_{s,*} =0 \quad \text{ and } \quad    
\|e^*_h\|_{H^s_*}=\frac1{(2\pi|h|)^s}\,     
\left(\sum_{\ell=0}^s(2\pi|h|)^{2\ell}\right)^{1/2}.     
$$     
Hence     
\begin{equation} \label{eq:appendix1}    
\abs{\il e^*_h, e^*_h\ir_{s,*}} \,=\, \|e^*_h\|_{H^s_*}^2 \,\le\,     
\frac1{1-(2\pi|h|)^{-2}}     
\,\le\, \frac{4\pi^2}{4\pi^2-1}.     
\end{equation}    
    
\bigskip    
    
To treat the case where $b_j, b_m$ are in $\{B^*_j\}_{j=0.1,\dots s}$,    
we need the following known properties of the normalized Bernoulli    
polynomials    
\begin{eqnarray*}     
\il B^*_m,B^*_0\ir_0&=& 1 \ \ \mbox{for $m=0$\  and \ $0$ for $m\ge1$},\\     
\il B_m^*,B^*_j\ir_0&=&(-1)^{\min\{m,j\}-1}\,B^*_{m+j}(0),     
\ \ \ \mbox{for all $m,j\ge1$}\\     
B^*_0(0)&=&1,\\     
B_{2m}^*(0)&=&\frac{2(-1)^{m+1}}{(2\pi)^{2m}}\ \zeta(2m)\ \ \     
\mbox{for all $m\ge1$},\\     
B^*_{2m+1}(0)&=&0\ \ \      
\mbox{for all $m\ge1$}.     
\end{eqnarray*}     
Here, as always, $\zeta$ is the Riemann zeta function.     

{}From these properties for $m\in[0,s]$ we conclude       
$$     
\|B^*_m\|_{H^s_*}^2=     
\sum_{\ell=0}^s\|[B^*_m]^{(\ell)}\|^2_{L_2}=     
\sum_{\ell=0}^m\|B^*_{m-\ell}\|^2_{L_2}=     
1+2\sum_{\ell=1}^m     
\frac{\zeta(2\ell)}{(2\pi)^{2\ell}}.     
$$     
Hence     
$$     
\|B^*_m\|_{H^s_*}^2\le      
1+\frac{2\zeta(2)}{(2\pi)^2}     
+\frac{2\zeta(4)}{4\pi^2(4\pi^2-1)}.     
$$     
Since $\zeta(2)=\pi^2/6$ and $\zeta(4)=\pi^4/90$, we conclude that     
\begin{equation}\label{eq:appendix2}     
\|B^*_m\|_{H^s_*}^2\le 1+\tfrac1{12}+\frac{\pi^2}{180(4\pi^2-1)}.     
\end{equation}     
We now consider $\il B^*_m,B^*_j\ir_{s,*}$     
for all $m,j\in[0,s]$ and $m\not= j$. Let    
$m'=\max\{m,j\}$ and  $j'=\min\{m,j\}$.    
Furthermore, let $\kappa_{m'-j'}=0$ for odd $m'-j'$     
and $\kappa_{m'-j'}=1$ for even  $m'-j'$. Then    
$$     
\il B^*_m,B^*_0\ir_{s,*}=\il B^*_m,B^*_0\ir_0=\delta_{m,0},     
$$     
whereas for $m,j\in[1,s]$ and $m\not=j$ we have      
\begin{eqnarray*}      
\il B^*_m,B^*_j\ir_{s,*}&=&\sum_{\ell=0}^{j'-1}\il     
B^*_{m-\ell},B^*_{j-\ell}\ir_0=     
\sum_{\ell=0}^{j'-1}(-1)^{j'-\ell-1}B^*_{m+j-2\ell}(0)\\    
&=&    
2\kappa_{m'-j'}(-1)^{(m'-j')/2}     
\sum_{\ell=0}^{j'-1}    
\frac{\zeta(m+j-2\ell)}     
{(2\pi)^{m+j-2\ell}}\\     
&=&     
2\kappa_{m'-j'}(-1)^{(m'-j')/2}     
\sum_{\ell=1}^{j'}    
\frac{\zeta(m'-j'+2\ell)}     
{(2\pi)^{m'-j'+2\ell}}.     
\end{eqnarray*}     
Note that the smallest argument of $\zeta$ for even $m'-j'$ is $4$.     
Therefore      
$$    
|\il B^*_m,B^*_j\ir_{s,*}|\le     
2\zeta(4)\,\frac{\kappa_{m'-j'}}{(2\pi)^{m'-j'}}     
\,\sum_{\ell=1}^\infty\frac1{(2\pi)^{2l}}\le     
\frac{2\zeta(4)}{4\pi^2-1}\,\frac{\kappa_{m'-j'}}{(2\pi)^{m'-j'}}.     
$$    
{}From this we have    
$$    
\sum_{m=0}^s|\il B^*_m,B^*_j\ir_{s,*}|\le    
\|B^*_j\|_{H^s_*}^2 +    
\frac{2\zeta(4)}{4\pi^2-1}\,   
\sum_{\substack{m=0\\m\neq j}}^s\frac{\kappa_{m'-j'}}{(2\pi)^{m'-j'}}.    
$$    
The non-zero terms correspond to even $m'-j'$ and each non-zero term    
$\kappa_{m'-j'}/(2\pi)^{m'-j'}$ may appear at most twice. Therefore     
\begin{equation}\label{eq:appendix3}     
\sum_{m=0}^s|\il B^*_m,B^*_j\ir_{s,*}|\le     
\|B^*_j\|_{H^s_*}^2 +    
\frac{4\zeta(4)}{4\pi^2-1}\,\sum_{\ell=0}^\infty    
\frac{1}{(2\pi)^{2\ell}}   
=\|B^*_j\|_{H^s_*}^2 + \frac{4\pi^4}{90(4\pi^2-1)^2}.    
\end{equation}     
We now consider the coefficients $\il B_j^*,e^*_h\ir_s$     
for $j=0,1,\dots,s$ and $h\in\Z\setminus0$.    
We start by showing that     
we have      
$$     
\int_0^1B_j^*(x)\,e^{-2\pi i\,hx}\,{\rm d}x=     
\begin{cases}\ 0\ \ &\ \ \ \mbox{for}\ \ j=0,\\     
\ \frac{-1}{(2\pi i\,h)^j}\ \ &\ \ \ \mbox{for}\ \ j\ge1.      
\end{cases}     
$$     
For $j=0$, it is zero since the integral of     
$e_h^*$ is zero for $h\neq0$.    
     
We now use induction on $j$. For $j=1$, we have     
$B_1^*(x)=x-\tfrac12$ and using integration     
by parts we get       
$$     
\int_0^1\left(x-\tfrac12\right)\,e^{-2\pi i\,hx}\,{\rm d}x=     
\frac{-1}{2\pi i\,h}     
\int_0^1\left(x-\tfrac12\right)\,{\rm d}\,e^{-2\pi i\,hx}=     
\frac{-1}{2\pi i\,h}\left(1-\int_0^1e^{-2\pi i\,hx}\,{\rm d}x\right)=     
\frac{-1}{2\pi i\,h},     
$$     
as claimed.      
     
For $j>1$, we again use integration by parts,     
and the property of Bernoulli polynomials~\eqref{ber0} and~\eqref{ber2}      
to obtain      
\begin{eqnarray*}     
\int_0^1B_j^*(x)\,e^{-2\pi i\,hx}\,{\rm d}x&=&     
\frac{-1}{2\pi i\,h}     
\int_0^1B_j^*(x)\,{\rm d}\, e^{-2\pi i\,hx}     
=     
\frac{1}{2\pi i\,h}\,     
\int_0^1[B_j^*]^{\prime}(x)\,e^{-2\pi i\,hx}\,     
{\rm d}x\\     
&=&     
\frac{1}{2\pi i\,h}\,\int_0^1B_{j-1}^*(x)\,     
e^{-2\pi i\ hx}\,{\rm d}x=\frac{-1}{(2\pi i\,h)^j},     
\end{eqnarray*}     
as claimed.     
     
{}From this we conclude that for $\ell\le j$ and $j\not=0$,     
$$     
\il [B_j^*]^{(\ell)},[e^{2\pi i\,h\cdot}]^{(\ell)}\ir_0=     
(-2\pi i\,h)^{\ell}\,\il B^*_{j-\ell},e^{2\pi i\,h\cdot}\ir_0=     
-(2\pi h)^{2\ell-j}\,i^{-j}.     
$$      
Clearly, for $j=0$ we have     
$\il [B_j^*]^{(\ell)},[e^{2\pi i\,h\cdot}]^{(\ell)}\ir_0=0$.     
     
Then for $h\neq0$ we have       
\begin{eqnarray*}     
\il B_j^*,e^*_h\ir_{s,*}&=&     
\sum_{\ell=0}^{s}     
\il [B_j^*]^{(\ell)},[e^*_h]^{(\ell)}\ir_0=     
\frac1{(2\pi|h|)^s}\,\sum_{\ell=0}^{j-1}     
\il [B_{j}^*]^{(\ell)},[e^{2\pi i\,h\cdot}]^{(\ell)}\ir_0\\     
&=&     
\begin{cases}     
\ 0\ \ &\ \ \mbox{for}\ \ \ j=0,\\     
\ \frac{-1}{(2\pi|h|)^s (2\pi i h)^j}\,     
\sum_{\ell=0}^{j-1}(2\pi h)^{2\ell}     
\ \ &\ \ \mbox{for}\ \ \ j\ge1.     
\end{cases}     
\end{eqnarray*}     
This allows us to compute      
$\abs{\il B_j^*,e^*_h\ir_{s,*}}$ for $h\neq0$ and $j=0,\dots,s$.     
For $j=0$ it is zero, and for $j\ge1$ we have     
\begin{equation}\label{eq:appendix4}    
\abs{\il B_j^*,e^*_h\ir_{s,*}} \,=\, \frac{1}{(2\pi |h|)^{s+j}}\     
\frac{(2\pi h)^{2j}-1}{(2\pi h)^2-1}     
\,<\, \frac{(2\pi |h|)^{j-s}}{(2\pi h)^2-1}    
\,\le\, \frac{4\pi^2}{4\pi^2-1}\ \frac1{(2\pi |h|)^{s-j+2}}.     
\end{equation}    
    
\bigskip    
    
We now are ready to bound $M_s$ from \eqref{eq:appendix_Ms}.    
For this we define    
\[    
M_{s,1} \,:=\, \max_{h\in\Z\setminus0}\left\{ \abs{\il e^*_h, e^*_h\ir_{s,*}}     
        \,+\, \sum_{j=0}^s \abs{\il B_j^*,e^*_h\ir_{s,*}} \right\}    
\]    
and    
\[    
M_{s,2} \,:=\, \max_{j=0,\dots,s}\left\{ \sum_{h\in\Z\setminus0} \abs{\il B_j^*,e^*_h\ir_{s,*}}     
        \,+\, \sum_{m=0}^s \abs{\il B_m^*,B_j^*\ir_{s,*}} \right\}    
\]    
such that $M_s=\max\{M_{s,1}, M_{s,2}\}$.     
    
Clearly, $M_s$ is strictly larger than one     
since the term inside of the maximum of     
$M_{s,2}$ corresponding to $j=0$ equals one.     
We now show that $M_s$ is close to one.     
    
{}From \eqref{eq:appendix1} and \eqref{eq:appendix4} we get    
\[\begin{split}    
M_{s,1} \;&\le\; \frac{4\pi^2}{4\pi^2-1}\,\left( 1     
        \,+\, \sum_{j=1}^s \frac1{(2\pi)^{s-j+2}} \right)     
\;\le\; \frac{4\pi^2}{4\pi^2-1}\,\left( 1     
        \,+\, \frac{1}{(2\pi)^2-2\pi} \right)     
\;\le\; 1.057.    
\end{split}\]    
We now bound $M_{s,2}$. We use \eqref{eq:appendix4} to bound     
the first sum in $M_{s,2}$,    
\[\begin{split}    
\sum_{h\in\Z\setminus0} \abs{\il B_j^*,e^*_h\ir_{s,*}}    
\;&\le\; \frac{4\pi^2}{4\pi^2-1}\, \sum_{h\in\Z\setminus0}    
\frac1{(2\pi |h|)^{s-j+2}}    
\;=\; \frac{8\pi^2}{4\pi^2-1}\, \frac{\zeta(s-j+2)}    
{(2\pi)^{s-j+2}}    
\\    
&\le\; \frac{8\pi^2}{4\pi^2-1}\,\frac{\zeta(2)}{(2\pi)^{2}}     
\;=\; \frac{\pi^2}{3(4\pi^2-1)}    
\;\le\; 0.0855.    
\end{split}\]

Using \eqref{eq:appendix2} and \eqref{eq:appendix3} we can bound     
the second sum in $M_{s,2}$ by    
$$    
\sum_{m=0}^s \abs{\il B_m^*,B_j^*\ir_{s,*}}     
\;\le\;     
 \tfrac{13}{12}+\frac{\pi^2}{180(4\pi^2-1)}     
+\frac{4\pi^4}{90(4\pi^2-1)^2}    
\;\le\; 1.0877.    
$$    
This shows that $M_{s,2}\le1.1732$ and, consequently, $M_s\le1.1732$    
and $\sqrt{M_s}\le    
1.0832\le\tfrac{13}{12}$. {}From~\eqref{eq:appendix_Ms}    
we finally obtain     
\[    
\|f\|_{H^s_*} \,\le\,\sqrt{M_s}\,\|f\|_{H^s}\le \tfrac{13}{12}\, \|f\|_{H^s}    
\]    
for all $f\in H^s$, as claimed

{\bf Acknowledgement.}       \ 
We thank Daan Huybrechs and two anonymous referees  
for valuable comments.

\end{document}